\documentclass[10pt]{article}
\usepackage{amsmath}
\usepackage[T1]{fontenc}
\usepackage[utf8]{inputenc}
\usepackage{amsfonts}
\usepackage{amssymb}
\usepackage[english]{babel}
\usepackage{graphicx}
\usepackage{amsthm}
\usepackage{hyperref}
\usepackage{accents}
\usepackage [autostyle, english = american]{csquotes}
\MakeOuterQuote{"}
\numberwithin{equation}{section}
\theoremstyle{plain}
\newtheorem*{theorem*}{Theorem}
\newtheorem*{lemma*}{Lemma}
\newtheorem{theorem}{Theorem}
\newtheorem{lemma}{Lemma}[section]
\newtheorem{corollary}[lemma]{Corollary}
\newtheorem{question}[lemma]{Question}

\newtheorem{proposition}[lemma]{Proposition}

\theoremstyle{definition}
\newtheorem{definition}[lemma]{Definition}
\newtheorem{remark}[lemma]{Remark}

\newcommand{\tet}{\theta}
\newcommand{\be}{\begin{equation}}
\newcommand{\ee}{\end{equation}}
\newcommand{\ba}{\begin{array}{l}}
\newcommand{\ea}{\end{array}}
\newcommand{\Rr}{{\mathbb R}}

\newcommand{\R}{{\mathcal{R}}}

\newcommand{\eps}{{\varepsilon}}

\newcommand{\Up}{\mathcal{U}^\phi}

\renewcommand{\div}{{\mbox{div}\,}}

\def\pr{\partial}

\usepackage{geometry}
\begin{document}

\title{On the Stability of Self-similar Blow-up for $C^{1,\alpha}$ Solutions to the Incompressible Euler Equations on $\Rr^3$}
\author{Tarek M. Elgindi\footnote{Department of Mathematics, UC San Diego. E-mail: telgindi@ucsd.edu.},~~Tej-Eddine Ghoul\footnote{Department of Mathematics, NYU Abu-Dhabi. E-mail: teg3@nyu.edu.}, and Nader Masmoudi\footnote{Department of Mathematics, NYU Abu-Dhabi and Courant Institute of Mathematical Sciences. E-mail: nm30@nyu.edu.}}
\date{}
\maketitle

\begin{abstract}
We study the stability of recently constructed self-similar blow-up solutions to the incompressible Euler equation. A consequence of our work is the existence of finite-energy $C^{1,\alpha}$ solutions that become singular in finite time in a locally self-similar manner. As a corollary, we also observe that the Beale-Kato-Majda criterion \emph{cannot} be improved in the class of $C^{1,\alpha}$ solutions. 
\end{abstract}

\tableofcontents

\section{Introduction}

\subsection{The Euler equation}

Recall the incompressible Euler equation governing the motion of an ideal fluid on $\mathbb{R}^3$:
\begin{equation}\label{E}\partial_t u+u\cdot\nabla u+\nabla p=0, \end{equation} 
\begin{equation}\label{incompressibility} \div(u)=0, \end{equation}
\begin{equation}\label{IC} u|_{t=0}=u_0. \end{equation}
$u:\mathbb{R}^3\times [0,\infty) \rightarrow\mathbb{R}^3$ is the velocity field of the fluid. $p$ is the force of internal pressure which acts to enforce the incompressibility constraint \eqref{incompressibility}. The incompressibility constraint \eqref{incompressibility} ensures that no patch of fluid can be compressed into a region of smaller volume. The incompressibility constraint has lead many to believe that loss of regularity for classical solutions is unlikely to occur, since concentration is not allowed. In fact, a quantitative consequence of the incompressibility constraint is that localized solutions which are $C^{1}$ in space and time on $\mathbb{R}^3\times [0,T)$ conserve their energy: \begin{equation}\label{EnergyBalance}\frac{d}{dt}\int_{\mathbb{R}^3} |u(x,t)|^2dx=0\end{equation} for all $t\in [0,T)$. Unfortunately, the conservation of total kinetic energy in the fluid does not seem to be enough to deduce that solutions to \eqref{E}-\eqref{IC} retain their regularity for all time as it does not preclude a blow-up of the gradient of the velocity field. This is indeed what happens in the Burgers equation in any dimension (which is \eqref{E} with $p\equiv 0$ and without the constraint \eqref{incompressibility}). On the other hand, the incompressibility constraint does prevent blow-up in two dimensions.  This is due to presence of higher order conservation laws, which will be discussed in the coming section. In the class of localized $C^\infty$ solutions, it remains a major open problem whether finite-time blow-up can happen on $\mathbb{R}^3$. In this work we are concerned with finite-energy $C^{1,\alpha}$ solutions to \eqref{E}-\eqref{IC}. Recently, self-similar solutions to \eqref{E}-\eqref{IC} were constructed in \cite{E_Classical}. This was done by showing that, in certain scenarios, the Euler equation \eqref{E}-\eqref{IC} can be viewed as a perturbation of a model equation with \emph{stable} self-similar blow-up. 

\subsection{The vorticity equation}
An important quantity to consider when studying ideal fluids is the vorticity vector field \[\omega:= \nabla\times u.\] It satisfies the vorticity equation:
\begin{equation}\label{Vorticity} \partial_t\omega+(u\cdot\nabla)\omega=(\omega\cdot \nabla) u.\end{equation}
Since $\div(u)=0$ we have that $\nabla \times (\nabla \times u)=-\Delta u.$ Thus, $u$ can be recovered from $\omega$ by the so-called Biot-Savart law:
\begin{equation}\label{BSLaw} u=(-\Delta)^{-1}(\nabla\times \omega).\end{equation} For classical solutions (with $u\in C^{1,\alpha}$ or, equivalently, $\omega\in C^\alpha$ for some $\alpha>0$), solving \eqref{E}-\eqref{incompressibility} is equivalent to solving \eqref{Vorticity}-\eqref{BSLaw} (so long as the vorticity is taken to be initially divergence-free when solving \eqref{Vorticity}-\eqref{BSLaw}). It is important to remark that when the velocity depends only on two coordinates, it is easy to show that solutions are globally regular and that the vorticity is finite for all finite time. This means that a singularity must come from a genuinely three-dimensional solution.

\subsection{Statement of the Main Theorem}
We now move to discuss the main result of this paper. 
\begin{theorem}\label{MainTheorem}
There is a continuum of $\alpha>0$ for which there exists a divergence-free $u_0\in C^{1,\alpha}(\mathbb{R}^3)$ with compactly supported initial vorticity $\omega_0\in C^\alpha(\mathbb{R}^3)$ so that the unique local solution to \eqref{E}-\eqref{IC} belonging to the class $L^2\cap C^{1,\alpha}_{x,t}([0,1)\times\mathbb{R}^3)$ satisfies \[\lim_{t\rightarrow 1}\int_0^t|\omega(s)|_{L^\infty}ds=+\infty.\] Moreover, the blow-up is stable in a sense that is specified in Theorem \ref{StabilityTheorem}.
\end{theorem}

\begin{remark}
The proof proceeds by showing that the self-similar solution constructed in $\cite{E_Classical}$ is stable with respect to perturbations in a space that allows for the full solution to be compactly supported. In fact, the perturbations are allowed to have non-trivial swirl.
\end{remark}

\begin{remark}
If we take $\alpha$ smaller and smaller, the blow-up becomes more and more mild. In particular, a consequence of our result is the following corollary.
\end{remark}

\begin{corollary}
In the class of all $L^2\cap C^{1,\alpha}$ solutions to the 3D Euler equations, it is not possible to strengthen the Beale-Kato-Majda criterion in the scale of $L^p$ spaces. In particular, for every $p<\infty$, there exists a classical solution to the 3D Euler equation for which 
\[\sup_{t\in[0,T_*)} \|\omega\|_{L^p}<\infty \qquad \text{while} \qquad \lim_{t\rightarrow T_*} \int_0^t \|\omega\|_{L^\infty}=+\infty. \]
\end{corollary}

\subsection{Discussion of the Result and its Proof}

Our work here proceeds from the point of view of asymptotic stability of stationary solutions in basic dynamical systems. In \cite{E_Classical}, a purely self-similar blow-up profile for the 3D Euler equation was constructed. That is, the vorticity satisfies:
\[\bar\omega(x,t)=\frac{1}{1-t}\bar G(\frac{x}{(1-t)^{\gamma_{*}}}),\] for some constant $\gamma_*$. In particular, $\bar G$ satisfies the static equation:
\[\bar G+\gamma_*z\cdot\nabla \bar G+u_{\bar G}\cdot\nabla \bar G=\bar G\cdot \nabla u_{\bar G},\] where 
\[u_{\bar G}=(-\Delta)^{-1}\nabla\times \bar G\] and $z=\frac{x}{(1-t)^{\gamma_*}}$. 
The self-similar profile can be viewed as a particular solution of the Euler equation in rescaled variables that we will now attempt to explain. For now, let $\tilde\lambda(t), \tilde\mu(t)$ be arbitrary functions of time (how they are chosen will be discussed later). We write:
\[\omega(x,t)=\frac{1}{\tilde\lambda(t)} G(\frac{x \tilde\mu(t)}{\tilde\lambda(t)^{\gamma_*}},t).\] As long as $\tilde\lambda,\tilde\mu$ are nice enough, it is clear that one can do this for any solution locally in time. This means that the Euler equation can be rewritten as:
\begin{equation}\label{SSModel}\frac{1}{\tilde\lambda(t)}\partial_t G-\frac{\tilde\lambda'(t)}{\tilde\lambda(t)^2} G-\frac{1}{\tilde\lambda(t)}\Big(\gamma_*\frac{\tilde\lambda'(t)}{\tilde\lambda(t)}+\frac{\tilde\mu'(t)}{\tilde\mu(t)}\Big)z\cdot\nabla G+\frac{1}{\tilde\lambda(t)^2}u_G\cdot\nabla G=\frac{1}{\tilde\lambda(t)^2}G\cdot \nabla u_{G}\end{equation}
Observe that 
\begin{equation}\label{SSExactSolution} (\tilde\lambda(t),\tilde\mu(t),G(\cdot,t))=((1-t),1,\bar G)\end{equation} is an exact solution to \eqref{SSModel}. Now, it is natural to ask the following 
\begin{question}
Is there a sense in which the blow-up solution \eqref{SSExactSolution} of \eqref{SSModel} is stable?
\end{question}
This is the question that we are concerned with. Recall that $\bar G$ is axi-symmetric without swirl. Thus, there are several levels at which this question can be asked: first, within the class of axi-symmetric solutions without swirl; second, within the class of general axi-symmetric solutions; finally, among general solutions to the 3D Euler equation. For this work, we content ourselves with answering the first and second questions. Studying the stability with respect to general 3D perturbations seems to be a difficult problem. 
\\

\subsubsection*{Linearized Operator and Modulation}
To answer the above question, we essentially have to linearize around the base solution \eqref{SSExactSolution} and study the behavior of the linearized operator. Once we are working in a setting where the linearized operator is coercive, just like in basic ODE theory, we should expect stability. Observe that the equation \eqref{SSModel} is underdetermined and so $\tilde\lambda(t)$ and $\tilde\mu(t)$ are used to keep perturbations within spaces where we have coercivity. This is similar to how the (different) parameters $\lambda$ and $\mu$ were used in \cite{E_Classical} to construct $\bar F$ in the first place. Indeed, if one just considers the relevant linearized operator for only axi-symmetric solutions \emph{without} swirl, then the linear estimates were essentially done in \cite{E_Classical}. This already allows us to assert Theorem \ref{MainTheorem} with relative ease. 

When we consider more general perturbations (such as perturbations with non-trivial swirl), then the linear and non-linear estimates become more difficult. In fact, the stability of the profile with respect to perturbations with swirl is only due to a matching of constants that shows that the linearized operator coming from the equation for the swirl is not worse than that of the axial vorticity. That the linear growth from the swirl is weaker than that from the axial vorticity is not obvious. The fact that this is the case is only due to the exact structure of the equation for the swirl and the nature of the coupling between the swirl and the axial vorticity. 

\subsection{Previous Works}

There are numerous works on the local and global well-posedness of the incompressible Euler equation in dimensions $d\geq 2$ and the blow-up problem.  We refer the reader to \cite{MB,Gibbon2008,BardosTitiReview,ConstantinReview,KiselevReview,HouLuo,EJE, EJB} for more in-depth reviews of the history of the singularity problem and related issues regarding the Euler equation. For the purpose of this discussion, we will only briefly discuss the issue of self-similar blow-up for the Euler equation. For the Euler and Navier-Stokes equations, the vast majority of the literature on self-similar blow-up was devoted to ruling out their existence \cite{Chae2007,ChaeShv,Tsai1998,NRS96}. Generally, these works have assumed the existence of a self-similar solution with (relatively) \emph{rapid} decay at spatial infinity. This seems to have been motivated by a desire to get self-similar solutions that themselves have finite energy. It seems that these assumptions were too restrictive. One purpose of our work is to emphasize that the lack of decay of the self-similar profile itself is \emph{not} an indication that the blow-up is "coming from infinity" nor is it impossible to get finite energy solutions from unbounded purely self-similar ones. We remark that there are also numerous works on "forward" self-similar solutions, where the data is already singular. Such solutions are also very interesting and can be used to prove results related to instability and non-uniqueness in various settings \cite{Elling, JiaSverakSS, VishikNon1, VishikNon2}.

Next, let us comment on the issue of "stability of blow-up." This work does not give a full picture of the stability question in our setting, but a partial result is given. Indeed, a natural question one could ask is: 
\begin{question}If $\omega_0$ gives rise to a blow-up, is it true that there is blow-up for the Euler equation for any data in a small\footnote{In a natural topology.} neighborhood of $\omega_0?$ 
\end{question}
The answer to this question remains open for the solution constructed in \cite{E_Classical}, but we do show an even stronger result for general axi-symmetric solutions that are odd in $x_3$ and are close to $\omega_0$ in a weighted Sobolev space (the weight basically imposes that perturbations are vanish to high order near the point of blow-up). Another way to think of the condition is that perturbations should be more regular than $\omega_0$ itself. Removing some of these conditions on the perturbation seems to be an interesting and challenging problem since many of the arguments rely heavily on the imposed symmetries. Some partial results in the negative direction were given in \cite{VasseurVishik} (though the notion of stability there seems to be quite stringent). 


Aside from the Euler equation, the issue of stable self-similar blow-up has been addressed many times before in other contexts. Indeed, our proof makes use of modulation techniques that have been developed by Merle, Raphael, Martel, Zaag and others. This technique has been very efficient to describe the blowup  the nonlinear wave equation \cite{MR3302641}, the nonlinear heat equation \cite{MR1427848}, reaction diffusion systems \cite{MR3779644,MR3846237}, the nonlinear Schrodinger equation \cite{MR2150386,MR2257393}, the GKDV equation \cite{MR3179608}, the Burgers equation \cite{CGM}, and many others.
Note that for 3D Euler comparing to all the previous models cited above there exists a group of scaling transformations of dimension larger than two that leaves the equation invariant and the incompressibility induces a nontrivial nonlocal effect. Here this degeneracy is a real difficulty since one does not know in advance which scaling law the flow will select.

We remark, finally, that after the completion of this work we came to know of the nice work of Chen and Hou \cite{ChenHou}, where many of the methods used in \cite{E_Classical} were adapted to the setting of the numerical work of Luo and Hou \cite{HouLuo}. They also established a form of stability that allows for compactly supported vorticity as is done here. Though the terminology used in \cite{ChenHou} is slightly different from that of \cite{E_Classical} and ours here, it seems that the methods are quite similar (notwithstanding technical differences due to the difference of the setting).

\subsection{Organization of the Paper}

In Section \ref{Setup} we recall the axi-symmetric Euler equation and setup the problem we intend to solve in this paper. In Section \ref{Coercivity}, we discuss the coercivity of the linearized operators. In Section \ref{DerivationOfTheLaws} we derive the "laws" that the modulation parameters $\mu$ and $\lambda$ should satisfy. In Section \ref{EllipticEstimates} we discuss the elliptic estimates that we need. In Section \ref{Final} we give the final energy estimate from which the main result follows. The appendix collects a few useful tools \ref{ProductRules}. Sections \ref{Setup}, \ref{Coercivity}, and \ref{SwirlParts} are the heart of the matter.

\subsection{Notation}\label{Notation}

In this subsection we give a guide to the notation used in the rest of the paper. 

\subsubsection*{Functions, variables, and parameters}
With the exception of introductory parts of this work, $r$ will generally denote the two dimensional radial variable: \[r=\sqrt{x_1^2+x_2^2}.\]
$\theta$ will denote\footnote{Except in Section 2.1-2.2 where it can also be taken to denote the two dimensional polar angle.} the angle between $r$ and $x_3$: \[\theta=\arctan(\frac{x_3}{r}),\] so that $\theta=0$ corresponds to the plane $x_3=0$ while $\theta=\pm \frac{\pi}{2}$ corresponds to the $x_3$ axis. $\rho$ will denote the three dimensional radial variable \[\rho=\sqrt{r^2+x_3^2}.\] $R$ will denote $\rho^\alpha$: \[R=\rho^\alpha\] (where $\alpha>0$ is a constant which will be small). Because the axial vorticity will be odd in the third variable, the $\theta$ variable will generally be in $[0,\pi/2]$ while the $R$ (later called $y$ or $z$) variable will usually be in $[0,\infty)$.  
The main parameters we will use are:
\[\eta=\frac{99}{100}, \qquad \alpha>0,\qquad \gamma=1+\frac{\alpha}{10}.\] $\alpha$ will be chosen at the end to be very small. 
In the later sections we use the functions \[\Gamma(\theta)=(\sin(\theta)\cos^2(\theta))^{\alpha/3}\] and \[K(\theta)=3\sin(\theta)\cos^2(\theta).\] Sometimes there will be a constant $c$ associated to $\Gamma$ written as $\frac{\Gamma}{c}$. This constant $c$ always satisfies $\frac{1}{10}\leq c\leq 10$ and it is a normalization constant. 

\subsubsection*{Norms and Operators}
We define the H\"older spaces using the norms:
\[|f|_{C^{\beta}(\mathcal{K})}=\sup_{x\in \mathcal{K}} |f|+\sup_{x\not=y} \frac{|f(x)-f(y)|}{|x-y|^\beta},\]
\[|f|_{C^{1,\beta}(\mathcal{K})}=|f|_{C^\beta}+|\nabla f|_{C^\beta}.\] When the domain $\mathcal{K}$ is clearly understood from context, we often omit writing it.

{\emph{ Warning}:} In most of this paper, we will be working in some form of polar or spherical coordinates and will be using spaces like $L^2([0,\infty)\times [0,\pi/2])$ or similar spaces where the relevant variables are a radial and angular variable. The norm on this space is the usual $L^2$ norm with the measure $dr d\theta$ and not the measure $r dr d\theta$. 

We define the weights \[w(z)=\frac{(1+z)^2}{z^2},\]  \[w_\theta(\theta)=\frac{1}{\sin(2\theta)^\frac{\gamma}{2}},\]  and \[W=w \cdot w_\theta.\]
We also define the differential operators:
\[D_\theta(f)=\sin(2\theta)\partial_\theta f, \qquad D_R(f)=R\partial_R f,\] and
\[D_z(f)=z\partial_z f.\]

For each $k\in\mathbb{N}$ we define the spaces $\mathcal{H}^k$ and $\mathcal{W}^{k,\infty}$ using the following norms. 
We define the $\mathcal{H}^k([0,\infty)\times[0,\pi/2])$ norm: \begin{equation}\label{HsNorm} |f|_{\mathcal{H}^k}^2=\sum_{i=0}^k|(D_R)^if \frac{w}{\sin^{\eta/2}(2\theta)}|_{L^2}^2+\sum_{i\geq 1, 1\leq i+j\leq k}|(D_\theta)^iD_R^jfW|^2_{L^2}.\end{equation} 
We also define the $\mathcal{W}^{l,\infty}$ norm:
\[|f|_{\mathcal{W}^{l,\infty}}=\sum_{0\leq k+j\leq l}|(z+1)^k\partial_z^k \Big(\frac{\sin(2\theta)}{\gamma-1+\sin(2\theta)}\partial_\theta\Big)^j f \sin(2\theta)^{-\frac{\alpha}{5}}|_{L^\infty}.\] 

In Section \ref{Coercivity} we inductively define an inner product on $\mathcal{H}^k$ which gives a norm equivalent to the $\mathcal{H}^k$ norm (with equivalence constant independent of $\alpha>0$). This inner-product is used to get coercivity out of the linearized operator $\mathcal{M}_F$ defined below in \eqref{LinearisedOpeMeps}. We remark that since there will be four linearized operators associated to $\eps$, $\mathcal{U}^\phi$,  $\partial_\theta\mathcal{U}^\phi$, and $\tan(\theta)\mathcal{U}^\phi$, we will actually be using four different inner products all defining norms equivalent to the $\mathcal{H}^k$ norm. For the first two, see Section \ref{Coercivity} and for the second two, see Section \ref{SwirlParts}.

\begin{remark}
It is clear that any smooth function vanishing at $0$ and $\pi/2$ and with sufficient $z$ decay belongs to $\mathcal{W}^{l,\infty}$ due to the inequality:
\[\sup_{x\in [0,1], \alpha\in [0,1]}\frac{x^{1-\frac{\alpha}{5}}}{\frac{\alpha}{10}+x}\leq 1.\] The basic example of a $\mathcal{W}^{l,\infty}$ function is the function \[\Gamma(\theta) \frac{z}{(1+z)^2}.\]
\end{remark}
Finally, define
\begin{definition}
 Let the integral operator $L_{12}: L^2([0,\infty)\times [0,\pi/2])\rightarrow L^2([0,\infty))$ be \[L_{12}(f)(z)=\int_{z}^\infty 3\int_0^{\pi/2}\frac{f(r,\theta)\sin(\theta)\cos^2(\theta)}{r}d\theta dr.\] 
\end{definition}

\section{The Setup}\label{Setup}
In this section we discuss the general setup and strategy that we will follow. We will first make a change of variables on the axi-symmetric Euler equation without swirl as in \cite{E_Classical}. Next, we will introduce similarity variables and the modulation parameters and show that the perturbation $\varepsilon$ from the purely self-similar solution decays exponentially. This is similar to the authors' previous work \cite{EGM} but with several added difficulties since the linearized problem is more delicate.

\subsection{Axi-symmetric Euler} 

We start with the axi-symmetric 3D incompressible Euler equations :
\begin{align} 
\partial_t u^\phi +u^r\pr_r u^\phi+ u^3\pr_{x_3} u^\phi =- \frac{u^\phi u^r}{r} & \\ 
\partial_t \omega  +u^r\pr_r\omega+ u^3\pr_{x_3} \omega = \frac{2u^\phi \pr_{x_3}u^\phi}{r}+\frac{\omega u^r}{r}\\ 
\end{align} 

where $(u^r,u^3)$ is determined as follows. 
First we solve the elliptic problem\footnote{Note that the $-$ sign on the left hand side is not conventionally added, but there is no difference up to a change of variables.}:
\[\partial_r(\frac{1}{r}\partial_r\tilde\psi)+\frac{1}{r}\partial_{33}\tilde\psi=-\omega\] and then we set \[u_r=\frac{1}{r}\partial_3\tilde\psi \qquad u_3=-\frac{1}{r}\partial_r\tilde\psi.\] 
Next, in order to fix the homogeneity, we set $\tilde\psi=r\psi$. 

Then we have:
\[u_r=\partial_3\psi\qquad u_3=-\frac{1}{r}\psi-\partial_r\psi\] and 
\[\partial_r(\frac{1}{r}\partial_r(r\psi))+\partial_{33}\psi=-\omega,\] which leads us to the system:
\begin{align}\label{axisymmetric3DE}
&\pr_t u^\phi+u_r\partial_r u^\phi+u_3\partial_3 u^\phi=-\frac{1}{r}u_r u^\phi\\
&\partial_t \omega+u_r\partial_r\omega+u_3\partial_3\omega=\frac{1}{r}u_r \omega+2\frac{u^\phi\pr_3 u^\phi}{r},
\end{align}
\begin{equation}\label{3dBSL}
-\partial_{rr}\psi-\partial_{33}\psi-\frac{1}{r}\partial_r\psi+\frac{\psi}{r^2}=\omega,
\end{equation}
\begin{equation}\label{PsiTou}
u_r=\partial_3\psi\qquad u_3=-\frac{1}{r}\psi-\partial_r\psi.
\end{equation}
The problem is normally set on the spatial domain $\{(r,x_3)\in [0,\infty)\times (-\infty,\infty)\}$ and  the elliptic problem \eqref{3dBSL} is solved with the boundary condition $\psi=0$ on $r=0$. We will start by imposing an odd symmetry on $\omega$ with respect to $x_3$. That is, we search for solutions with: \[\omega(r,x_3)=-\omega(r,-x_3)\] for all $r,x_3$. Consequently, we may reduce to solving on the domain $[0,\infty)\times[0,\infty)$ while enforcing that $\psi$ vanish on $r=0$ and $x_3=0$ when solving \eqref{3dBSL}:
\begin{equation} \label{BCForPsi} \psi(r,0)=\psi(0,x_3)=0, \end{equation} for all $r,x_3\in [0,\infty)$. 

\subsection{Passing to a form of polar coordinates}

First we define $\rho=\sqrt{r^2+x_3^2}$ and $\theta=\arctan(\frac{x_3}{r})$ and set $R=\rho^\alpha$ for some (small) constant $\alpha>0$. Then we introduce new functions $\omega(r,x_3)=\Omega(R,\theta)$, $u^\phi=\rho U^\phi(R,\tet)$ and $\psi(r,x_3)=\rho^2 \Phi_\Omega (R,\theta)$. We now show the forms of \eqref{axisymmetric3DE}, \eqref{3dBSL}, and \eqref{PsiTou} in the new coordinates.
Note that 
\[
\partial_r\rightarrow  \frac{\cos(\theta)}{\rho}\alpha R \partial_R-\frac{\sin(\theta)}{\rho} \partial_\theta \qquad \partial_3\rightarrow \frac{\sin(\theta)}{\rho}\alpha R\partial_R+\frac{\cos(\theta)}{\rho} \partial_\theta
\]

\subsubsection*{$u$ in terms of $\Phi_\Omega$}
From \eqref{PsiTou} and the above facts we see:
\[u^r=\rho\Big(2\sin(\theta)\Phi_\Omega+\alpha\sin(\theta)R\partial_R\Phi_\Omega+\cos(\theta)\partial_\theta\Phi_\Omega\Big)\] while \[u^3=\rho\Big(-\frac{1}{\cos(\theta)}\Phi_\Omega-2\cos(\theta)\Phi_\Omega-\alpha\cos(\theta)R\partial_R\Phi_\Omega+\sin(\theta)\partial_\theta\Phi_\Omega\Big)\]

\subsubsection*{Evolution equation for $\Omega$ and $U^\phi$}
Observe that using the above calculations, \eqref{axisymmetric3DE} becomes 
\begin{align}\label{OmegaEvolution}
&\pr_t U^\phi+U(\Phi_\Omega)\pr_\tet U^\phi+V(\Phi_\Omega)\alpha R\pr_R U^\phi=-\mathcal{R}(\Phi_\Omega)U^\phi-V(\Phi_\Omega)U^\phi,\\
&\partial_t\Omega+U(\Phi_\Omega)\partial_\theta\Omega+V(\Phi_\Omega)\alpha R\partial_R\Omega=\mathcal{R}(\Phi_\Omega)\Omega+\frac{2}{\cos\tet}U^\phi(\sin(\tet)\alpha R\pr_R+\cos(\tet)\pr_\tet)U^\phi,
\end{align}
where
\begin{align}
&U(\Phi_\Omega)=-3\Phi_\Omega-\alpha R\partial_R\Phi_\Omega,\quad V(\Phi_\Omega)=\partial_\theta\Phi_\Omega-\tan(\theta)\Phi_\Omega,\\
&\mathcal{R}(\Phi_\Omega)=\frac{1}{\cos(\theta)}\Big(2\sin(\theta)\Psi+\alpha\sin(\theta)R\partial_R\Psi+\cos(\theta)\partial_\theta\Psi\Big).
\end{align}

\subsubsection*{Relation between $\Phi_\Omega$ and $\Omega$}
After some calculations\footnote{See the calculation preceding \eqref{PolarBSL}.} \eqref{3dBSL} becomes: \begin{equation}\label{PolarBSL1}-\alpha^2R^2\partial_{RR}\Phi_\Omega-\alpha(5+\alpha)R\partial_R\Phi_\Omega-\partial_{\theta\theta}\Phi_\Omega+\partial_\theta\big(\tan(\theta)\Phi_\Omega\big)-6\Phi_\Omega=\Omega.\end{equation} with the boundary conditions:
\[\Phi_\Omega(R,0)=\Phi_\Omega(R,\frac{\pi}{2})=0\] for all $R\in [0,\infty)$. 

\subsection{Self-similar variables}
Indeed it is shown in \cite{E_Classical} that there exists a self-similar solution of the form for the vanishing swirl system ($U^\phi=0$):
\[\Omega=\frac{1}{T-t}F\Big(\frac{R}{(T-t)^{1+\delta}},\theta\Big)\] where $\delta$ is a small real number depending on $\alpha$. Recall that $F=F_*+\alpha^2 g$, where \[F_*=F_*(\alpha)=\frac{\Gamma}{c}\frac{4\alpha z}{(1+z)^2},\qquad |g|_{\mathcal{H}^k}\leq C,\] with $C$ a constant independent of $\alpha$.  We introduce the self-similar variable \[z=\frac{R}{(T-t)^{1+\delta}}.\] It is easy to see that if $\Omega$ has the above form, then $\Phi_\Omega$ should have the form: \[\Phi_\Omega=\frac{1}{T-t}\Phi_F(z,\theta).\] Now we write the equations for $F$ and $\Phi_F$:

\begin{align}\label{EquationProfileF}
F+(1+\delta) z\partial_z F +U(\Phi_F)\partial_\theta F+\alpha V(\Phi_F) z\partial_z F=\mathcal{R}(\Phi_F)F 
\end{align}
\[U(\Phi_F):=-3\Phi_F-\alpha R\partial_R\Phi_F\quad V(\Phi_F):=\partial_\theta\Phi_F-\tan(\theta)\Phi_F, \quad \R(\Phi_F):=\frac{1}{\cos(\theta)}\Big(2\sin(\theta)\Phi_F+\alpha\sin(\theta)R\partial_R\Phi_F+\cos(\theta)\partial_\theta\Phi_F\Big),\]
\begin{align}\label{EquationProfilePhi}
-\alpha^2R^2\partial_{RR}\Phi_F-\alpha(5+\alpha)R\partial_R\Phi_F-\partial_{\theta\theta}\Phi_F+\partial_\theta\big(\tan(\theta)\Phi_F\big)-6\Phi_F=F.
\end{align}

To prove the stability of the profiles $(F,\Phi_F)$ we rescale \eqref{OmegaEvolution} and \eqref{PolarBSL1}.
A natural change of variables to do here will be 
\begin{align}\label{changeofvariable}
z&=\frac{R}{\lambda^{1+\delta}},\quad\quad\frac{ds}{dt}=\frac{1}{\lambda},\nonumber\\
 \Omega(R,t,\theta)&=\frac{1}{\lambda}\Xi\Big(\frac{R}{\lambda^{1+\delta}},s,\theta \Big),~~\Phi_\Omega(R,t,\theta)=\frac{1}{\lambda}\Phi_\Xi\Big(\frac{R}{\lambda^{1+\delta}},s,\theta \Big)~~ U^\Phi(R,t,\theta)=\frac{1}{\lambda}\tilde{U}^\phi\Big(\frac{R}{\lambda^{1+\delta}},s,\theta \Big).
\end{align}

Note that $\tilde{F}=F(\mu z,\theta)$, $\tilde{\Phi}_F=\Phi_F(\mu\cdot)$ is also a solution of \eqref{EquationProfileF} and \eqref{EquationProfilePhi}. This scaling invariance on \eqref{EquationProfileF} and \eqref{EquationProfilePhi} will induce an instability later on the linearized operator around $(F,\Phi_F)$.
To fix this instability, we introduce a new parameter $\mu:=\mu(t)$ and fix it through an orthogonality condition.
Hence, we introduce
$$\Xi(z)=W(\mu z),~~\Phi_\Xi=\Phi_W(\mu z),~~\mathcal{U}^\phi=\tilde{U}^\phi(\mu z)~~y=\mu z$$
where $(W,\mathcal{U}^\phi,\Phi_W)$ solves
\begin{align}\label{FirstscalingEvolution3}
&\mathcal{U}^\phi_s+\frac{\mu_s}{\mu}y\pr_y\mathcal{U}^\phi-\frac{\lambda_s}{\lambda}\mathcal{S}_\delta(\mathcal{U}^\phi) +U(\Phi_W)\partial_\theta \mathcal{U}^\phi+V(\Phi_W)\alpha y\partial_y \mathcal{U}^\phi=-\mathcal{R}(\Phi_W)\mathcal{U}^\phi-V(\Phi_W)\mathcal{U}^\phi,\\
&W_s+\frac{\mu_s}{\mu}y\pr_y W-\frac{\lambda_s}{\lambda}\mathcal{S}_\delta(W) +U(\Phi_W)\partial_\theta W+V(\Phi_W)\alpha y\partial_y W=\mathcal{R}(\Phi_W)W+2\mathcal{U}^\phi(\tan(\tet)\alpha y\pr_y+\pr_\tet)\mathcal{U}^\phi,
\end{align}
\begin{align}\label{FirstscalingEvolution4}
-\alpha^2z^2\partial_{zz}\Phi_W-\alpha(5+\alpha)z\partial_z\Phi_W-\partial_{\theta\theta}\Phi_W+\partial_\theta\big(\tan(\theta)\Phi_W\big)-6\Phi_W=W.
\end{align}
and
$$\mathcal{S}_\delta(W)=W+(1+\delta)y\pr_y W.$$
Now we linearize around $(F,0,\Phi_F)$ by setting,
$$W=F+\varepsilon,~~\Phi_W=\Phi_F+\Phi_\varepsilon,~~\mathcal{U}^\phi=\mathcal{U}^\phi+0.$$                                     
Hence, we obtain the following equation
\begin{align}\label{Equationepsilon}
\left\{\begin{array}{lll}
&\mathcal{U}^\phi_s+\frac{\mu_s}{\mu}y\pr_y\mathcal{U}^\phi-\Big(\frac{\lambda_s}{\lambda}+1\Big)\mathcal{S}_\delta(\mathcal{U}^\phi)+\mathcal{M}^\phi_F\mathcal{U}^\phi=N_1(\Phi_\eps,\mathcal{U}^\phi),\\
&\pr_s \eps+\frac{\mu_s}{\mu}y\pr_y \eps-\Big(\frac{\lambda_s}{\lambda}+1\Big)\mathcal{S}_\delta(\eps)+\mathcal{M}_F\eps=E+N_2(\eps)+N_3(\mathcal{U}^\phi),\\
&-\alpha^2y^2\partial_{yy}\Phi_\eps-\alpha(5+\alpha)z\partial_z\Phi_\eps-\partial_{\theta\theta}\Phi_\eps+\partial_\theta\big(\tan(\theta)\Phi_\eps\big)-6\Phi_\eps=\eps.
\end{array}
 \right.
 \end{align}
where $\mathcal{M}_F$ and $\mathcal{M}_F^\phi$ are the linearized operators given by
\begin{align}\label{LinearisedOpeMeps}
\mathcal{M}_F\eps=\mathcal{S}_\delta(\eps)+U(\Phi_F)\pr_\tet\eps +V(\Phi_F)\alpha y\pr_y \eps+U(\Phi_\eps)\pr_\tet F +V(\Phi_\eps)\alpha y\pr_y F-\mathcal{R}(\Phi_F)\eps-\mathcal{R}(\Phi_\eps)F,
\end{align}
\begin{align}\label{LinearisedopeU}
\mathcal{M}^\phi_F\Up=\mathcal{S}_\delta(\Up)+U(\Phi_F)\pr_\tet\Up +V(\Phi_F)\alpha y\pr_y \Up+(\mathcal{R}(\Phi_F)+V(\Phi_F))\Up,
\end{align}

$E$ is the error 
\begin{align}\label{Error}
E=-\frac{\mu_s}{\mu}y\pr_y F+\Big(\frac{\lambda_s}{\lambda}+1\Big)\mathcal{S}_\delta(F),
\end{align}
and the non-linear terms,
\begin{align}\label{nonlinearterms}
&N_1(\Phi_\eps,\Up)=-U(\Phi_\eps)\pr_\tet \Up -V(\Phi_\eps)\alpha y\pr_y \Up-(\mathcal{R}(\Phi_\eps)+V(\Phi_\eps))\Up,\\
&N_2(\eps)=-U(\Phi_\eps)\partial_\theta \eps-\alpha V(\Phi_\eps) y\partial_y \eps+\mathcal{R}(\Phi_\eps) \eps,\\
&N_3(\Up)=2\mathcal{U}^\phi(\tan(\tet)\alpha y\pr_y+\pr_\tet)\mathcal{U}^\phi.
\end{align}



We will allow $\mu$ and $\lambda$ to depend on $s$ to be able to fix $\pr_y\eps(0,\theta)=L_{12}(\eps)(0)=0$ for all $\theta$. The reason that we wish to keep this information on $\varepsilon$ is that this is precisely what will allow us to squeeze some damping out of the linearized operator $\mathcal{M}_F$. Also note that we will need $\pr_y\Up(y=0,\tet,s)=0$ for all $\tet\in[0,\frac{\pi}{2}]$ and $s\geq0$. This is propagated once we assume it initially. Note that even though the condition $\partial_y\eps(0,\theta)=0$ seems to require $\mu$ and/or $\lambda$ to depend on $\theta$,  the important property of the equation is that, once $L_{12}(\eps)(0)=0$, we have that $\Phi_\eps(0,\theta)=0$ for all $\theta$. Since all non-linear terms are roughly of the form $\eps\Phi_\eps$, the quadratic vanishing is propagated once we have that $L_{12}(\eps)(0)=0$. 

\subsubsection{The emergence of $L_{12}$ and the role of $F_*$}

One important fact that we will use in our analysis is that the solution $\Phi$ of the third equation in \eqref{Equationepsilon}, can be written as:
\[\Phi_\eps= \frac{1}{4\alpha}\sin(2\theta)L_{12}(\eps)+\bar\Phi_\eps,\] where \[|\partial_{\theta}^2\bar\Phi_\eps|_{\mathcal{H}^k}+\alpha|D_R\partial_\theta\bar\Phi_\eps|_{\mathcal{H}^k}+\alpha^2|D_R^2\bar\Phi_\eps|_{\mathcal{H}^k}\leq C|\eps|_{\mathcal{H}^k}\] (see Theorem \ref{Elliptic}). Consequently, we see from this that 
\begin{equation}\label{Approximation1} U(\Phi_\eps)=-3\Phi_\eps-\alpha D_y\Phi_\eps= -\frac{3}{4\alpha}\sin(2\theta)L_{12}(\eps)+O(1) \end{equation}
\begin{equation}\label{Approximation2} V(\Phi_\eps)=\partial_\theta\Phi_\eps-\tan(\theta)\Phi_\eps= \frac{1}{4\alpha}(2\cos(2\theta)-2\sin^2(\theta))L_{12}(\eps)+O(1)  \end{equation}
\begin{equation}\label{Approximation3} \mathcal{R}(\Phi_\eps)=\frac{1}{\cos(\theta)}(2\sin(\theta)\Phi_\eps+\alpha\sin(2\theta)D_y\Phi_\eps+\cos(\theta)\partial_\theta\Phi_\eps)= \frac{1}{2\alpha}L_{12}(\eps)+O(1),  \end{equation} where the $O(1)$ terms above are terms involving $\bar\Phi_\eps$ which satisfies bounds independent of $\alpha$. 

Note also that since $F=F_*+\alpha^2 g$ with \[F_*=\frac{\Gamma}{c}\frac{4\alpha y}{(1+y)^2},\] we have that:
\begin{equation}\label{ApproximationF1} U(\Phi_F)= -3\sin(2\theta)\frac{1}{1+y} +O(\alpha),\end{equation}
\begin{equation}\label{ApproximationF2} V(\Phi_F)= (2\cos(2\theta)-2\sin^2(\theta))\frac{1}{1+y} +O(\alpha), \end{equation}
\begin{equation}\label{ApproximationF3} \mathcal{R}(\Phi_F)= \frac{2}{1+y}+O(\alpha).  \end{equation}
It will be helpful to keep these approximations in mind when studying the leading order behavior of $\mathcal{M}_F$ and $\mathcal{M}_F^\phi$. 

\subsection{General Strategy}

As explained in the beginning of the section, our goal will now be to use some coercivity from the terms $\mathcal{M}_F(\varepsilon)$ and $\mathcal{M}_F^\phi(\mathcal{U}^\phi)$ to prove:
\[\frac{d}{dt} \bar{\mathcal{E}}\leq -c\bar{\mathcal{E}}+C\bar{\mathcal{E}}^{3/2},\] for some constants $c,C>0$. This will then show, with a suitable bootstrap argument (as in Section 3.1 of \cite{EGM}), that if $\mathcal{E}(\varepsilon_0)$ is sufficiently small, we have that $\mathcal{E}(\varepsilon)$ decays exponentially as $s\rightarrow \infty$. The focus will now be to prove coercivity estimates on $\mathcal{M}_F$ as well as the relevant elliptic and product estimates that will enable us to establish the above. The consequence is the following stability theorem from which Theorem \ref{MainTheorem} and its Corollary follow.

\subsection{Stability Theorem}
\begin{definition}[Energy]
Fix $k\geq 4.$ For $\eps\in \mathcal{H}^k$ and $\mathcal{U}^\phi\in\mathcal{H}^{k+1}$ with $\tan(\theta)\mathcal{U}^\phi\in\mathcal{H}^k,$ define \[\mathcal{E}(\eps,\mathcal{U}^\phi)=\| \eps\|_{\mathcal{H}^k}+\|\Up\|_{\mathcal{H}^{k+1}}+\|\pr_\tet\Up\|_{\mathcal{H}^{k}}+\|\tan(\tet)\Up\|_{\mathcal{H}^{k}}.\]

\end{definition}
\begin{theorem}\label{StabilityTheorem}
For $k\geq 4$, there exists $\alpha_0>0$ small so that for all $\alpha<\alpha_0$, there is a $\delta_0>0$ and $\kappa>0$ so that for every initial $(\eps_0, \mathcal{U}^\phi_0)$ with $\mathcal{E}(\eps_0,\mathcal{U}^\phi_0)<\delta_0 \alpha^{3/2}$ and $L_{12}(\eps_0)(0)=0$, there is an associated unique global solution to \eqref{Equationepsilon} so that:
\[|\mu_s|+|\frac{\lambda_s}{\lambda}+1|+\mathcal{E}(\eps,\mathcal{U}^\phi)(s)\leq C\mathcal{E}(\eps_0,\mathcal{U}^\phi_0)e^{-\kappa s}\] for all $s\geq 0$. 
\end{theorem}

\begin{corollary}
Under the conditions of Theorem \ref{StabilityTheorem}, \[\mu(s)\rightarrow \mu_{\infty}\] exponentially as $s\rightarrow\infty$ and there exists $T_*$ so that \[\lambda(s)\exp(s)\rightarrow \frac{1}{T_*}\approx 1,\] exponentially fast as $s\rightarrow\infty$. Consequently, $\frac{T_*}{T_*-t}\lambda(t)\rightarrow 1$ as $t\rightarrow T_*$. \end{corollary}

\begin{remark}
Note that this corollary follows directly from Theorem \ref{StabilityTheorem} in view of the scaling laws \eqref{changeofvariable}.
\end{remark}
\subsection{Solutions with compactly supported vorticity and finite energy}
In view of Theorem \ref{StabilityTheorem}, to get finite-time singularity for compactly supported solutions, it suffices to show that there exists $\eps_0\in\mathcal{H}^k$ with small norm so that $F+\eps_0$ is compactly supported. We take $\mathcal{U}^\phi_0$ to be compactly supported. This is not difficult to do since $F$ decays faster than $\eps_0$ needs to. Indeed, let $\chi\in C_c^\infty([0,\infty))$ with $\chi\equiv 1$ on $[0,1]$, $\chi\equiv 0$ on $[2,\infty)$, and $0\leq \chi\leq 1.$ We will let $M>>1$ and $\beta<<1$ be positive constants to be chosen later. Consider
\[\eps_0^{M,\beta}=(\chi(\frac{z}{M})-1)F+\beta\sin(2\theta)\chi((z-3)^2).\] Observe that $\eps_0+F$ is compactly supported. Next, observe that \[\|\eps_0^{M,\beta}\|_{\mathcal{H}^k}\leq \frac{C}{M^{1/4}}+C\beta+C\alpha^2.\] This is just due to the fact that $F=F_*+\alpha^2 g$ and $F_*\approx \alpha z^{-1}$ as $z\rightarrow\infty$ and $|g|_{\mathcal{H}^k}\leq C$, while the $\mathcal{H}^k$ norm is like an $L^2$ norm for large $z$. We also use Lemma \ref{RadialMultiplier}. Note also that \[L_{12}(\eps_0^{M,\beta})(0)=L_{12}(\chi(Mz)-1)(0)+\beta L_{12}(\sin(2\theta)\chi( (z-3)^2))(0).\]
Observe that $a_M:=L_{12}(\chi(Mz)-1)(0)$ satisfies:
\[|a_M|\leq \frac{C}{M}+C\alpha^2,\] for the same reason that $F=F_*+\alpha^2 g$. 
On the other hand, the fixed constant $b=L_{12}(\sin(2\theta)\chi((z-3)^2))(0)>0$. 
Thus we define \[\beta=-\frac{a_M}{b}.\] Then we have:
\[L_{12}(\eps^{M,\beta})(0)=0\qquad \|\eps_0^{M,\beta}\|_{\mathcal{H}^k}\leq C\Big(\frac{1}{M^{1/4}}+\alpha^2+\beta\Big)\leq \frac{C}{M^{1/4}}+C\alpha^2\] if $M$ is large. In particular, if we take $M= \frac{1}{\alpha^8}$ and if $\alpha$ is sufficiently small, $\eps_0^{M,\beta}$ will satisfy the hypothesis of Theorem \ref{StabilityTheorem}. 

\section{Coercivity in $\mathcal{H}^k$} \label{Coercivity}
Recall from \cite{E_Classical} the definition of the following operators.
\begin{definition}\label{Definition_Linearization}
 \[\mathcal{L}_{F_*}(f)= f+y\partial_y f-2 \frac{f}{1+y}-\frac{2y\Gamma(\theta)}{c(1+y)^2}L_{12}(f),\]
\[\mathcal{L}(f)=f+y\partial_y f-2\frac{f}{1+y},\]
$$\mathcal{L}^T_{F_*} f:=\mathcal{L}_{F_*}(f)+\frac{3}{1+y}\sin(2\tet)\pr_\tet f- \frac{\Gamma(\theta)}{c}\frac{2 z^2}{(1+z)^3}L_{12}(\frac{3}{1+y}\sin(2\tet)\pr_\tet f)(0).$$
\end{definition}
$\mathcal{L}_{F_*}$ is the linearization of the fundamental model around $F_*$, which is the leading order of the linearized operator from \cite{E_Classical}. 
The extra term $\frac{\Gamma(\theta)}{c}\frac{2 z^2}{(1+z)^3}L_{12}(\frac{3}{1+y}\sin(2\tet)\pr_\tet f)(0)$ is to ensure that  $L_{12}(\mathcal{L}^T_{F_*}(f))(0)=L_{12}(\mathcal{L}_{F_*}(f))(0)$.

\subsection{Coercivity of $\mathcal{L}_{F_*}^T$ in $\mathcal{H}^k$}
\label{LinearizedEstimates1}
We begin with the following proposition.
\begin{proposition}
There exists an inner-product on $\mathcal{H}^k$ that gives a norm equivalent to the $\mathcal{H}^k$ norm so that
\begin{equation}\label{Coercivity1}(\mathcal{L}_{F_*}^T f, f)_{\mathcal{H}^k}\geq |f|_{\mathcal{H}^k}^2,\end{equation} and \begin{equation}\label{Coercivity1}(\mathcal{L}_{F_*} f, f)_{\mathcal{H}^k}\geq |f|_{\mathcal{H}^k}^2,\end{equation} whenever $f\in\mathcal{H}^k$ and $L_{12}(f)(0)=0$.
\end{proposition}
 \begin{proof}
To do this, we proceed by induction on $k$. We know that $\mathcal{L}_{F_*}^T$ is coercive in $\mathcal{H}^2$ with a suitable inner product whose first term is always 10. Let us now show how to pass from coercivity on $\mathcal{H}^{k-1}$ to coercivity on $\mathcal{H}^k$. Note that $D_\theta$ commutes with $\mathcal{L}_{F_*}^T$ with the exception of the last term in $\mathcal{L}_{F_*}^T$. By the induction assumption, we assume that
\[(\mathcal{L}_{F_*}^T g, g)_{\mathcal{H}^{k-1}}\geq |g|_{\mathcal{H}^{k-1}}^2,\]
and \[(\mathcal{L}g, g)_{\mathcal{H}^{k-1}}\geq |g|_{\mathcal{H}^{k-1}}^2.\]
We first study
$(D_\theta^k \mathcal{L}_{F_*}^T f, D_\theta^k f W)_{L^2}$. Recall that \[D_\theta^{k}\mathcal{L}_{F_*}^T(f)=\mathcal{L}(D_\theta^k f)+D_\theta^k\Gamma \Big(-\frac{2y}{c(1+y)^2}L_{12}(f)+\frac{1}{c}\frac{2 y^2}{(1+y)^3}L_{12}(\frac{3}{1+y}\sin(2\theta)\partial_\theta f)(0)\Big)\] \[-\frac{3}{1+y}\sin(2\theta)\partial_\theta D_\theta^k f.\]
Using the definition of $W$, it is now easy to see that 
\[(D_\theta^k \mathcal{L}_{F_*}^T f, D_\theta^k f W)_{L^2}\geq \frac{1}{10}|D_\theta ^kf W|_{L^2}-C\alpha |f|_{\mathcal{H}^1}^2.\]
Now let us proceed by induction. Let us assume that for $ 0\leq j\leq k-2$ we have established an estimate of the form:
\[(D_\theta^{k-j}D_y^j \mathcal{L}_{F_*}^T f, D_\theta^{k-j} D_y^j f W)_{L^2}\geq \frac{1}{10}|D_\theta^{k-j}D_y^j  f W|_{L^2}^2- C|f|_{\mathcal{H}^{k-1}}^2,\] then we will show that \[(D_\theta^{k-(j+1)}D_y^{j+1} \mathcal{L}_{F_*}^T f, D_\theta^{k-(j+1)} D_y^{j+1} f W)_{L^2}\geq \frac{1}{100}|D_\theta^{k-(j+1)}D_y^{j+1}  f W|_{L^2}^2- C|f|_{\mathcal{H}^{k-1}}^2-C|D_\theta^{k-j}D_y^j  f W|_{L^2}^2.\]
Once this is established, we will be done by induction on $j$. We first apply $D_{\theta}^{k-(j_1)}$ to $\mathcal{L}_{F_*}^T$ and we get:
 \[D_\theta^{k}\mathcal{L}_{F_*}^T(f)=\mathcal{L}(D_\theta^{k-(j+1)} f)+D_\theta^{k-(j+1)}\Gamma \Big(-\frac{2y}{c(1+y)^2}L_{12}(f)+\frac{1}{c}\frac{2 y^2}{(1+y)^3}L_{12}(\frac{3}{1+y}\sin(2\theta)\partial_\theta f)(0)\Big)\]\[-\frac{3}{1+y}\sin(2\theta)\partial_\theta D_\theta^{k-(j+1)} f=I+II+III.\]
Now we apply $D_y^{j+1}$ to the above expression. First observe that \[(D_y^{j+1}I, D_y^{j+1} D_\theta^{k-(j+1)}f W)_{L^2}\geq \frac{1}{10}|D_\theta^{k-(j+1)}D_y^{j+1}  f W|_{L^2}^2-C|D_\theta^{k-(j+1)}D_y^{j+1}  f W|_{L^2}|f|_{\mathcal{H}^{k-1}}.\] This is because $D_y$ commutes with the derivative term in $\mathcal{L}$ and its commutator with the other terms is lower order. 
The term with $II$ is low order and we leave it to the reader. As for the term III:
\[(D_y^{j+1}III, D_y^{j+1} D_\theta^{k-(j+1)}f W)_{L^2}\leq 100\alpha |D_\theta^{k-(j+1)}D_y^{j+1}  f W|_{L^2}^2\]\[+C|D_\theta^{k-(j+1)}D_y^{j+1}  f W|_{L^2}(|D_\theta^{k-j}D_y^{j}  f W|_{L^2}+|f|_{\mathcal{H}^{k-1}})\] the first term comes from the term where all derivatives fall on $f$ in $III$ and we then integrate by parts using the definition of $W$. The first part of the second term comes when one $D_y$ hits the factor $\frac{3}{1+y}$ and the second part of the second term comes when more than one derivative hits that factor. This shows that we can find a suitable inner product whose norm is equivalent to the $\mathcal{H}^k$ norm, with equivalence constants independent of $\alpha$, so that \eqref{Coercivity1} holds. Note that for the linear estimates $\alpha$ does not need to be taken to be smaller as $k$ is taken larger. Note also that we have treated the term $\frac{\Gamma(\theta)}{c}\frac{2 y^2}{(1+y)^3}L_{12}(\frac{3}{1+y}\sin(2\tet)\pr_\tet f)(0)$ perturbatively (as a purely "bad" term) so we also have that 
\[(\mathcal{L}_{F_*}^T(f)+\frac{\Gamma(\theta)}{c}\frac{2 y^2}{(1+y)^3}L_{12}(\frac{3}{1+y}\sin(2\tet)\pr_\tet f)(0),f)_{\mathcal{H}^k}\geq |f|_{\mathcal{H}^k}^2.\] 
We also deduce for the same reason that
\[(\mathcal{L}_{F_*}(f),f)_{\mathcal{H}^k}\geq |f|_{\mathcal{H}^k}^2.\] 
\end{proof}

\subsection{Coercivity of $\mathcal{M}_F$ in $\mathcal{H}^k$}\label{MF}
Our goal in this section is to derive coercivity estimates for $\mathcal{M}_F$ given what we know about $\mathcal{L}_{F_*}^T$ from the previous section. The argument will be merely perturbative.
Indeed, the self-similar solution $F$ is a perturbation of $F_*=2\alpha\frac{\Gamma(\theta)}{c}\frac{y}{(1+y)^2}$
$$F=F_*+\alpha^2 g.$$
Hence, one can also write $\mathcal{M}_F$ as
\begin{align}
\mathcal{M}_F&=\mathcal{L}^T_{F_*}+\frac{2y^2\Gamma}{c(1+y)^3}L_{12}(\frac{3}{1+y}\sin(2\theta)\partial_\theta \varepsilon)(0))+\sqrt{\alpha}\tilde{L}
\end{align}
and $$\tilde{L}f= - \frac{1}{\sqrt{\alpha}}\Big[ \alpha V(F_*)y\partial_y\varepsilon+U(\Phi_{\varepsilon})\partial_\theta F_* +\alpha V(\Phi_\varepsilon)y\partial_y F_*+l.o.t.\Big].$$ We refer the reader to \eqref{ApproximationF1}-\eqref{ApproximationF3} to see how the computation above is done. In the above, "l.o.t." refers to lower order terms in $\alpha$. These are the terms coming from $g$ and their size is made precise in the following proposition.

\begin{proposition}\label{proposition:lowerorderL}
\[|(\tilde L(g),g)_{\mathcal{H}^k}|\leq C|g|_{\mathcal{H}^k}^2.\]
\end{proposition}
\begin{remark}
This is because of the $\sqrt{\alpha}$ loss in the product rules of Section \ref{ProductRules}. 
\end{remark}
\begin{proof}
This essentially follows from the computations from the preceding subsection. Indeed, using the notation of \cite{E_Classical}, $F=F_*+\alpha^2 g$ with $|g|_{\mathcal{H}^k}\leq C$ and $F_*$ is of order $\alpha$. Since $g$ is small, we can essentially discard the linear terms containing $g$ and focus on the rest. 
Now, by definition, we see that \[\mathcal{M}_F(\varepsilon)=\mathcal{L}_{F_*}^T(\varepsilon)- \frac{y^2\Gamma}{c(1+y)^3}L_{12}(\frac{3}{1+y}\sin(2\theta)\partial_\theta \varepsilon)(0))\] \[ +\alpha V(F_*)y\partial_y\varepsilon+U(\Phi_{\varepsilon})\partial_\theta F_* +\alpha V(\Phi_\varepsilon)y\partial_y F_*+l.o.t.\] 
The result now follows using Propositions \ref{P1}, \ref{P2}, \ref{T1}, and \ref{T2} as well as Theorem \ref{Elliptic}.
\end{proof}

Hence, the following proposition will follow from the coercivity of $\mathcal{L}_{F_*}^T$ in $\mathcal{H}^k$.
\begin{proposition}\label{proposition:coercivityM}
Let $\varepsilon\in \mathcal{H}^k$ satisfy that $L_{12}(\varepsilon)(0)=0$. There exists a constant $c$ depending only on $k$ so that if $\alpha$ is sufficiently small, we have that \begin{equation} (\mathcal{M}_F(\varepsilon),\varepsilon)_{\mathcal{H}^k}\geq c|\varepsilon|_{\mathcal{H}^k}^2.\end{equation}
\end{proposition}

\subsection{Coercivity of $\mathcal{M}_F^\phi$ in $\mathcal{H}^k$}\label{MPhi}
Consider $\mathcal{M}_F^\phi:$ 
\[\mathcal{M}^\phi_F f=\mathcal{S}_\delta(f)+U(\Phi_F)\pr_\tet f +V(\Phi_F)\alpha y\pr_y f+(\mathcal{R}(\Phi_F)+V(\Phi_F))f.\] In view of \eqref{ApproximationF1}-\eqref{ApproximationF3}, this gives:
\[\mathcal{M}^\phi_F (f)=y+y\partial_y f-\frac{3}{1+y}\sin(2\theta)\partial_\theta f+\frac{4-6\sin^2(\theta)}{1+y}f+l.o.t.,\]
\[=y+y\partial_y f-\frac{3}{1+y}\sin(2\theta)\partial_\theta f-\frac{2}{1+y}f+\frac{6\cos^2(\theta)}{1+y}f+l.o.t.\]
\[=\mathcal{L}_{F_*}^T+\frac{6\cos^2(\theta)}{1+y}f+l.o.t.\]
where the lower order terms are coming from the $g$ term in the expansion $F=F_*+\alpha^2 g$. It is then easy to see that since the extra $6\cos^2(\theta)$ term above has the right sign and we have that to leading order $\mathcal{M}_F^\phi$ is "more positive" than $\mathcal{L}_{F_*}^T$. It is then easy to show the following proposition. 

\begin{proposition}\label{proposition:coercivityMphi} 
For all $f\in\mathcal{H}^k$ we have that
\begin{equation}\label{Mphi1} (\mathcal{M}^\phi_F(f),f)_{\mathcal{H}^k_{U_0}}\geq c|f|_{\mathcal{H}^k}^2,\end{equation} where $(\cdot,\cdot)_{\mathcal{H}^k_{U_0}}$ is an inner product on $\mathcal{H}^k$ that gives rise to a norm equivalent to the $\mathcal{H}^k$ norm. 
\end{proposition}

\section{The bootstrap regime} \label{Bootstrap}

We will define first in which sense the solution is initial close to the self-similar profile. 

\begin{definition}[Initial closeness]\label{definition:ini} 
Let $\delta>0$ small enough, $s_0\gg 1$, and $W_0,~~\Up_0\in \mathcal{H}^k$. We say that $(W_0,\Up_0)$ is initially close to the blow-up profile $(F,0)$ if there exists $\lambda_0>0$ and $\mu_0>0$ such that the following properties are verified. In the variables $(y,s)$ one has:
\begin{align}
W_0(y)=F+\eps_0,~~\mathcal{U}^\phi_0
\end{align}
and the remainder and the parameters satisfy:
\begin{itemize}
\item[(i)] \emph{Initial values of the modulation parameters:}
\begin{align} \label{parametersini}
\frac 12 e^{-\frac{s_0}{2}}< \lambda_0 < 2 e^{-\frac{s_0}{2}}, \ \ \frac 12<\mu_0 < 2
\end{align}
\item[(ii)] \emph{Initial smallness:}
\begin{align} \label{qini}
\| W_0\|_{\mathcal{H}^2}^2+\|\mathcal{U}^\phi_0\|_{\mathcal{H}^2}^2+\|\pr_\tet\Up_0\|_{\mathcal{H}^{k}}^2+\|\tan(\tet)\Up_0\|_{\mathcal{H}^{k}}< e^{-\frac{s_0}{8}}, 
\end{align}
\end{itemize}
\end{definition}

We are going to prove that solutions initially close to the self-similar profile in the sense of Definition \ref{definition:ini} will stay close to this self-similar profile in the following sense.

\begin{definition}[Trapped solutions] \label{definition:trap}

Let $K\gg 1$. We say that a solution $w$ is trapped on $[s_0,s^*]$ if it satisfies the properties of Definition \ref{definition:ini} at time $s_0$, and if it can be decomposed as  
$$W=F+\eps$$
 for all $s\in [s_0,s^*]$ with:
\begin{itemize}
\item[(i)] \emph{Values of the modulation parameters:}
\begin{align} \label{parameterstrap}
\frac 1K e^{-\frac{s}{8}}< \lambda(s) < K e^{-\frac{s}{8}}, \ \ \frac 1K <\mu(s) < K.
\end{align}
\item[(ii)] \emph{Smallness of the remainder:}
\begin{align} \label{etrap}
\| \eps\|_{\mathcal{H}^k}+\|\Up\|_{\mathcal{H}^{k+1}}+\|\pr_\tet\Up\|_{\mathcal{H}^{k}}+\|\tan(\tet)\Up\|_{\mathcal{H}^{k}}< Ke^{-\frac s8},
\end{align}
\end{itemize}
\end{definition}

\begin{proposition} \label{proposition:bootstrap}
There exist universal constants $K,s_0^*\gg 1$ such that the following holds for any $s_0\geq s_0^*$. All solutions $w$ initially close to the self-similar profile in the sense of Definition \ref{definition:ini} are trapped on $[s_0,+\infty)$ in the sense of Definition \ref{definition:trap}.
\end{proposition}
Define for $\delta>0$ small enough:
\begin{align}\label{defenergy}
\mathcal{E}(s)=\| \eps\|_{\mathcal{H}^k}+\|\Up\|_{\mathcal{H}^{k+1}}+\|\pr_\tet\Up\|_{\mathcal{H}^{k}}+\|\tan(\tet)\Up\|_{\mathcal{H}^{k}}.
\end{align}

The proof of the proposition will be done later by using energy estimates. Before this we will derive that "law" that $\mu$ and $\lambda$ will satisfy.

\section{Derivation of the laws} \label{DerivationOfTheLaws}
In this section, we derive equations for $\mu$ and $\lambda$ that allow us to propagate $\partial_y \varepsilon(0,\theta,s)=L_{12}(\varepsilon)(0,s)=0$ for all $\theta$ and $s>0$. We assume that these hold at $s=0$. Do derive these laws, we first apply $L_{12}$ to the first equation of \eqref{Equationepsilon} and evaluate at $0$. It is helpful to observe the following facts: whenever $f(0,\theta)=0$ for all\footnote{Essentially all the functions we deal with have this property.}  $\theta$, we have \[L_{12}(y\partial_y f)(0)=0.\]

\begin{proposition}\label{proposition:derivationoflaw}
To keep $\pr_y\Up(0,\tet,s)=0$, $L_{12}(\varepsilon)(0,s)=0$ and $\partial_y\varepsilon(0,\theta,s)=0$ for all $\theta\in [0,\frac{\pi}{2}]$ and $s>0$ it suffices to impose that $\lambda$ and $\mu$ satisfy the following ODE's:
\begin{equation}\label{LambdaEquation}-\alpha\Big(\frac{\lambda_s}{\lambda}+1\Big)+3L_{12}(\frac{\sin(2\tet)}{1+y}\pr_\tet\eps)(0)=\sqrt{\alpha} L_{12}(\tilde{L}\eps)(0)+L_{12}(N_2(\eps))(0)+L_{12}(N_3(\Up))(0),\end{equation}
\begin{equation}\label{MuEquation}\frac{\mu_s}{\mu}=(2+\delta)\Big(\frac{\lambda_s}{\lambda}+1\Big).\end{equation}
$$\pr_y\Up(0,\tet)\Big|_{s=0}=0$$
We also have the following bounds,
\begin{align}
&\Bigg|\alpha\Big(\frac{\lambda_s}{\lambda}+1\Big)-3L_{12}(\frac{\sin(2\tet)}{1+y}\pr_\tet\eps)(0)\Bigg|\lesssim\sqrt{\alpha}\|\eps\|_{\mathcal{H}^2}+\|\Up\|_{\mathcal{H}^2}^2.\\
&\Big|\frac{\lambda_s}{\lambda}+1\Big|\lesssim\frac{1}{\alpha}\|\eps\|_{\mathcal{H}^2}
\end{align}
\end{proposition}
\begin{remark}
It is useful to review the contents of Section \ref{MF}, particularly the definition of $\mathcal{M}_F$ and its main terms for this calculation. 
\end{remark}
\begin{proof}
\begin{align}
&\pr_y\Up_s\Big|_{y=0}+\frac{\mu_s}{\mu}\pr_y(y\pr_y\Up)\Big|_{y=0}-\Big(\frac{\lambda_s}{\lambda}+1\Big)\pr_y(\mathcal{S}_\delta(\Up))\Big|_{y=0}=\pr_y\mathcal{M}^\phi_F\Up\Big|_{y=0}\nonumber\\
&+\pr_y N_1(\Phi_\eps,\Up)\Big|_{y=0}\\
& \pr_y\eps_s\Big|_{y=0}+\frac{\mu_s}{\mu}\pr_y(y\pr_y \eps)\Big|_{y=0}-\Big(\frac{\lambda_s}{\lambda}+1\Big)\pr_y\mathcal{S}_\delta(\eps)\Big|_{y=0}=\pr_y E\Big|_{y=0}+\pr_y\mathcal{M}_F\eps\Big|_{y=0}\nonumber\\
&+\pr_y N_2(\eps)\Big|_{y=0}+\pr_y N_3(\Up)\Big|_{y=0},\\
\end{align}
By using that
\begin{align}
\eps(s,0,\theta)=\Up(s,0,\theta)=\pr_y\eps(s,0,\theta)=L_{12}(\eps)(0)=0,
\end{align}
we deduce easily that,
\begin{align}
\pr_y\eps_s(s,0,\theta)=\pr_y(y\pr_y\Up)\Big|_{y=0}=\pr_y(y\pr_y\varepsilon)\Big|_{y=0}=\pr_y(\mathcal{S}_\delta(\varepsilon))\Big|_{y=0}=0.
\end{align}
Also from $L_{12}(\eps)(0)=0$ we deduce that,
\begin{align}\label{vanish0}
\Phi_\eps(0,\theta)=\pr_\tet\Phi_\eps(0,\tet)=0,
\end{align}
Hence, from \eqref{vanish0} we obtain that for all $\theta\in[0,\frac{\pi}{2}]$,
\begin{align}
U(\Phi_\eps)\Big|_{y=0}=V(\Phi_\eps)\Big|_{y=0}=\mathcal{R}(\Phi_\eps)\Big|_{y=0}=0.
\end{align}

Hence,
\begin{align}
\pr_y N_1(\Phi_\eps,\Up)\Big|_{y=0}=\pr_y N_2(\eps)\Big|_{y=0}=\pr_y\mathcal{M}_F\eps\Big|_{y=0}=0.
\end{align}

By using that $F=\alpha F_*+\alpha^2 g$ with $\pr_yg(0,\theta)=0$ for all $\theta\in[0,\frac{\pi}{2}]$ and $F_*=\frac{\Gamma(\theta)}{c}\frac{2y}{(1+y)^2}$ we deduce that
$$\pr_y E\Big|_{y=0}=-\frac{\mu_s}{\mu}\pr_y F(0,\theta)+\Big(\frac{\lambda_s}{\lambda}+1\Big)(2+\delta)\pr_y F(0,\theta)=\Big(-\frac{\mu_s}{\mu}+(2+\delta)\Big(\frac{\lambda_s}{\lambda}+1\Big)\Big)\alpha\frac{2\Gamma(\theta)}{c}.$$
It follows from the previous computations that 
$$-\frac{\mu_s}{\mu}+(2+\delta)\Big(\frac{\lambda_s}{\lambda}+1\Big)=0.$$
Similarly, from 
\begin{align}
\pr_y\Up_s\Big|_{y=0}+\frac{\mu_s}{\mu}\pr_y(y\pr_y\Up)\Big|_{y=0}-\Big(\frac{\lambda_s}{\lambda}+1\Big)\pr_y(\mathcal{S}_\delta(\Up))\Big|_{y=0}=\pr_y\mathcal{M}^\phi_F\Up\Big|_{y=0}+\pr_y N_1(\Phi_\eps,\Up)\Big|_{y=0}
\end{align}
 we deduce that
\begin{align}
{\pr_y\Up(0,\tet)_s}+\frac{\mu_s}{\mu}\pr_y\Up(0,\tet)-{\pr_y\Up(0,\tet)}\Big(\frac{\lambda_s}{\lambda}+1\Big)(2+\gamma)-\pr_y(\mathcal{M}^\phi_F(\pr_y\Up))\Big|_{y=0}=0.
\end{align}

In addition, from $\Phi_F=\frac{1}{4\alpha}\sin(2\theta)L_{12}(\alpha F_*)+\Phi_g$ with $\Phi_g(0)=0$ and $L_{12}(F_*)(0)=4\alpha$, we deduce that
\begin{align}
U(\Phi_F)(0)=-3\sin(2\tet),~~V(\Phi_F)(0)=3\cos(2\tet)-1,~~\mathcal{R}(\Phi_F)(0)=2.
\end{align}
Let's compute the linear terms $\pr_y\mathcal{M}^\phi_F(\pr_y\Up)\Big|_{y=0}$:
\begin{align}
\pr_y\mathcal{M}^\phi_F(\pr_y\Up)\Big|_{y=0}=\Big(\gamma-4\alpha+6(1+\alpha)\cos^2(\tet)\Big)\pr_y\Up(0,\tet)-3\sin(2\tet)\pr_\tet \pr_y\Up(0,\tet).
\end{align}

Hence,
\begin{align}
\pr_y\Up(0,\tet)_s+\Big(\gamma-4\alpha+6(1+\alpha)\cos^2(\tet)\Big)\pr_y\Up(0,\tet)-3\sin(2\tet)\pr_\tet \pr_y\Up(0,\tet)=0.
\end{align}
Since $\pr_y\Up(0,\tet)$ is transported through the previous equation it is clear that if $\pr_y\Up(0,\tet)$ was $0$ initially it will stay for all $s\geq0$.

To get the law on $\lambda$ we apply $L_{12}$ to the equation of $\eps$ in \eqref{Equationepsilon} and take the trace at $y=0$.
$$\pr_s L_{12}(\eps)(0)+\frac{\mu_s}{\mu}L_{12}(y\pr_y \eps)(0)-\Big(\frac{\lambda_s}{\lambda}+1\Big)L_{12}(\mathcal{S}_\delta(\eps))(0)=L_{12}(E)(0)-L_{12}(\mathcal{M}_F\eps)(0)+L_{12}(N_2(\eps))+L_{12}(N_3(\Up))(0).$$
We compute first $L_{12}(E)(0)$ by using that $F= F_*+\alpha^2 g$ with $L_{12}(g)(0)=0$,
$$L_{12}(E)(0)=\Big(\frac{\lambda_s}{\lambda}+1\Big)L_{12}( F_*)(0)=\alpha\Big(\frac{\lambda_s}{\lambda}+1\Big).$$
We use from Proposition \ref{proposition:lowerorderL} that
\begin{align}\label{Linearopdecomp}
\mathcal{M}_F\eps=\mathcal{L}_{F_*}\eps-\frac{3\sin(2\tet)}{(1+y)}\pr_\tet\eps+\sqrt{\alpha}\tilde{L}\eps,
\end{align}
where $\mathcal{L}_{F_*}\eps:=\eps+z\partial_z \eps-2 \frac{\eps}{1+z}-\frac{2z\Gamma(\theta)}{c(1+z)^2}L_{12}(\eps)$.
It follows that,
$$L_{12}(\mathcal{M}_F\eps)(0)= L_{12}(\mathcal{L}_{F_*}\eps)(0)-L_{12}(\frac{3\sin(2\tet)}{(1+y)}\pr_\tet\eps)(0)+\sqrt{\alpha} L_{12}(\tilde{L}\eps)(0).$$
To prove that some terms are zero we will use the following identity. 

 \begin{equation}\label{L12L}  L_{12}\Big(\mathcal{L}_{F_*}(f)\Big)=\mathcal{L}\Big(L_{12}(f)\Big)\end{equation}
 \begin{equation}\label{LW} \mathcal{L}(g)w=gw+z\partial_z(gw).\end{equation}
 where $\mathcal{L}\eps:=f+z\partial_z f-2 \frac{f}{1+z}$.
By using the previous identity and $\eps_y(0,\tet)=L_{12}\eps(0)=0$ we deduce that
$L_{12}(\mathcal{L}_{F_*}\eps)(0)=0$ as well as $L_{12}(\mathcal{S}_\delta(\eps))(0)=0$.
Finally we obtain the following second law,
\begin{align}
-\alpha\Big(\frac{\lambda_s}{\lambda}+1\Big)+3L_{12}(\frac{\sin(2\tet)}{1+y}\pr_\tet\eps)(0)=\sqrt{\alpha} L_{12}(\tilde{L}\eps)(0)+L_{12}(N_2(\eps))(0)+L_{12}(N_3(\Up))(0).
\end{align}
Hence, by using Proposition \ref{proposition:lowerorderL}, Proposition \ref{P1} and Proposition \ref{P2} we deduce that
\begin{align}\label{goodboundonlambda}
\Bigg|\alpha\Big(\frac{\lambda_s}{\lambda}+1\Big)-3L_{12}(\frac{\sin(2\tet)}{1+y}\pr_\tet\eps)(0)\Bigg|\lesssim\sqrt{\alpha}\|\eps\|_{\mathcal{H}^2}+\alpha\|\eps\|_{\mathcal{H}^2}^2+\|\Up\|_{\mathcal{H}^2}^2.
\end{align}
There is also a rough bound
\begin{align}\label{roughboundonlambda}
\Big|\frac{\lambda_s}{\lambda}+1\Big|\lesssim\frac{1}{\alpha}\Big|L_{12}(\frac{\sin(2\tet)}{1+y}\pr_\tet\eps)(0)\Big|\lesssim\frac{1}{\alpha}\|\eps\|_{\mathcal{H}^2}.
\end{align}

\end{proof}


\section{Elliptic Estimates}\label{EllipticEstimates}

We now prove elliptic estimates in all $\mathcal{H}^k$ spaces. This was done in the case $k=2$ in \cite{E_Classical}. We consider solutions to the following elliptic boundary value problem. 

Given $F$ satisfying that $(F, K)_{L^2_\theta}\equiv 0$, we solve
\begin{equation}\label{PolarBSL} -\alpha^2D_R^2\Psi-\alpha D_R\Psi-\partial_{\theta\theta}\Psi+\partial_\theta\big(\tan(\theta)\Psi\big)-6\Psi=F.\end{equation}
We couple this equation with the natural boundary conditions on $\Psi$:
\[\Psi(R,0)=\Psi(R,\pi/2)=0,\qquad \lim_{R\rightarrow\infty} \Psi(R,\theta)=0.\]
We will show that \[\alpha^2|D_R^2\Psi|_{\mathcal{H}^k}+\alpha|D_R \Psi|_{\mathcal{H}^k}+|\partial_{\theta}^2 \Psi|_{\mathcal{H}^k}\leq C_k |F|_{\mathcal{H}^k}.\]
This has already been established in the case $k=2$. Observe that $D_R$ commutes with \eqref{PolarBSL}, so this allows us to prove higher elliptic regularity estimates for the radial derivatives. 
Now let us rewrite \eqref{PolarBSL} as:
\[-\partial_{\theta\theta}\Psi+\partial_\theta(\tan(\theta)\Psi)=F+6\Psi+\alpha^2 D_R^2 \Psi+\alpha D_R\Psi:=G.\]
Since estimates on the radial derivatives are relatively simple to get, it suffices to establish $\mathcal{H}^k$ estimates on just the angular part of the equation: 
\[-\partial_{\theta\theta}\Psi+\partial_\theta(\tan(\theta)\Psi)=G,\] for $|G|_{\mathcal{H}^k}\leq C_k |F|_{\mathcal{H}^k}$. 
Now we wish to show that the following quantity is non-positive up to lower order terms \[(\partial_\theta^{k+1}(\tan(\theta)\Psi), \partial_\theta^{k+2} \Psi \sin(2\theta)^{2k-\gamma})_{L^2_\theta}.\]
It is natural to consider $\tilde\Psi= \frac{\Psi}{\cos(\theta)}$ so that we wish to study:
\[(\partial_\theta^{k+1}(\sin(\theta)\tilde\Psi), \partial_\theta^{k+2} (\cos(\theta)\tilde\Psi) \sin(2\theta)^{2k-\gamma})_{L^2_\theta}.\]
By induction on $k$, it suffices to consider only the following three terms: 
\[\frac{1}{2}\int \partial_\theta^{k+1}\tilde\Psi\partial_\theta^{k+2}\tilde\Psi \sin(2\theta)^{2k+1-\gamma}-(k+2)\int\sin^2(\theta)\big(\partial_\theta^{k+1}\tilde\Psi\big)^2\sin(2\theta)^{2k-\gamma}\] \[+(k+1)\int \cos^2(\theta)\partial_\theta^k\tilde\Psi \partial_\theta^{k+2}\tilde\Psi \sin(2\theta)^{2k-\gamma}\]
\[=\frac{2k+1-\gamma}{4}\int (\partial_\theta^{k+1}\tilde\Psi)^2 \cos(2\theta) \sin(2\theta)^{2k-\gamma}-(k+2)\int\sin^2(\theta)\big(\partial_\theta^{k+1}\tilde\Psi\big)^2\sin(2\theta)^{2k-\gamma}\] \[-(k+1)\int \cos^2(\theta)\big(\partial_\theta^{k+1}\tilde\Psi \big)^2sin(2\theta)^{2k-\gamma}+E,\] where $E$ is lower order and satisfies \[|Ew_r^2|_{L^1_R}\leq C|F|_{\mathcal{H}^k}.\]
On the other hand, we have:
\[\frac{2k+1-\gamma}{4}\int (\partial_\theta^{k+1}\tilde\Psi)^2 \cos(2\theta) \sin(2\theta)^{2k-\gamma}-(k+2)\int\sin^2(\theta)\big(\partial_\theta^{k+1}\tilde\Psi\big)^2\sin(2\theta)^{2k-\gamma}\] \[-(k+1)\int \cos^2(\theta)\big(\partial_\theta^{k+1}\tilde\Psi \big)^2sin(2\theta)^{2k-\gamma}\leq 0.\] 
By induction on $k$ we now have the following theorem.
\begin{theorem}\label{Elliptic}
Let $k\geq 2$ and assume $F\in\mathcal{H}^k$ satisfies $(F, K)_{L^2_\theta}\equiv 0$. Then,  \[\alpha^2|D_R^2\Psi|_{\mathcal{H}^k}+\alpha|D_R \Psi|_{\mathcal{H}^k}+|\partial_{\theta}^2 \Psi|_{\mathcal{H}^k}\leq C_k |F|_{\mathcal{H}^k}.\]

\end{theorem}

\section{Final Energy Estimates}\label{Final}

Recall the equation solved by $(\eps,\Up,\Phi_\eps)$:
\begin{align}
\left\{\begin{array}{lll}
&\Up_s+\frac{\mu_s}{\mu}y\pr_y\Up-\frac{1}{2}\Big(\frac{\lambda_s}{\lambda}+1\Big)\mathcal{S}_\delta(\Up)+\mathcal{M}^\phi_F\Up=N_1(\Phi_\eps,\Up)\\
& \eps_s+\frac{\mu_s}{\mu}y\pr_y \eps-\frac{1}{2}\Big(\frac{\lambda_s}{\lambda}+1\Big)\mathcal{S}_\delta(\eps)+\mathcal{M}_F\eps=E+N_2(\eps)+N_3(\Up),\\
&-\alpha^2y^2\partial_{yy}\Phi_\eps-\alpha(5+\alpha)z\partial_z\Phi_\eps-\partial_{\theta\theta}\Phi_\eps+\partial_\theta\big(\tan(\theta)\Phi_\eps\big)-6\Phi_\eps=\eps.
\end{array}
 \right.
\end{align}

As alluded to in the bootstrap section \ref{Bootstrap}, our goal will be to now control the following total energy:
 \[\mathcal{E}=\mathcal{E}(\eps,U^\phi)=|\eps|_{\mathcal{H}^k}+|U^\phi|_{\mathcal{H}^{k+1}}+|\partial_\theta U^\phi|_{\mathcal{H}^k}+|\tan(\theta)U^\phi|_{\mathcal{H}^k}.\] Strictly speaking, we only now have linear coercivity for $\eps$ and $U^\phi$ in $\mathcal{H}^k$ and $\mathcal{H}^{k+1}$ respectively; however, it turns out that because of the structure of the equations we will also get linear damping on the second to terms in the energy.  Our goal will be to now show that $\mathcal{E}$ decays exponentially if $\mathcal{E}_0$ is sufficiently small. In the coming sections we will compute $\frac{d}{ds}\mathcal{E}$ term by term. 

\subsection{Bound on $\frac{d}{ds}(\eps,\eps)_{\mathcal{H}^k}$}

We have that \[\frac{1}{2}\frac{d}{ds} (\eps,\eps)_{\mathcal{H}^k}\leq -(\mathcal{M}_F\eps,\eps)_{\mathcal{H}^k}+(E,\eps)_{\mathcal{H}^k}+\Big|\frac{\mu_s}{\mu}\Big||(y\partial_y\eps,\eps)_{\mathcal{H}^k}|+\Big|\frac{\lambda_s}{\lambda}+1\Big||(S_\delta(\eps),\eps)_{\mathcal{H}^k}|+(N_2(\eps),\eps)_{\mathcal{H}^k}+(N_3(U^\phi),\eps)_{\mathcal{H}^k}.\]
\subsubsection*{Coercivity from $\mathcal{M}_F$}
We first write:
\[-(\mathcal{M}_F\eps,\eps)_{\mathcal{H}^k}+(E,\eps)_{\mathcal{H}^k}=-(M_F(\eps)-\frac{\Gamma(\theta)}{c}\frac{2 y^2}{(1+y)^3}L_{12}(\frac{3}{1+y}\sin(2\tet)\pr_\tet \eps)(0),\eps)_{\mathcal{H}^k}\]\[+(E-\frac{\Gamma(\theta)}{c}\frac{2 y^2}{(1+y)^3}L_{12}(\frac{3}{1+y}\sin(2\tet)\pr_\tet \eps)(0),\eps)_{\mathcal{H}^k}\]
Now, from Proposition \ref{proposition:coercivityM}, we have that:
\[(\mathcal{M}_F\eps-\frac{\Gamma(\theta)}{c}\frac{2 y^2}{(1+y)^3}L_{12}(\frac{3}{1+y}\sin(2\tet)\pr_\tet \eps)(0),\eps)_{\mathcal{H}^k}\geq c|\eps|_{\mathcal{H}^k}^2.\]
\subsubsection*{Estimate of the Error Term}
In addition, recall that
\begin{align}
&E-\frac{\Gamma(\theta)}{c}\frac{2 y^2}{(1+y)^2}L_{12}(\frac{3}{1+y}\sin(2\theta)\partial_\theta\eps)(0)\nonumber\\
&= -\frac{\mu_s}{\mu}y\pr_y F+\Big(\frac{\lambda_s}{\lambda}+1\Big)\mathcal{S}_\delta(F)-\frac{\Gamma(\theta)}{c}\frac{2 y^2}{(1+y)^2}L_{12}(\frac{3}{1+y}\sin(2\theta)\partial_\theta\eps)(0).
\end{align}
By using $\frac{\mu_s}{\mu}=(2+\delta)(\frac{\lambda_s}{\lambda}+1)$ we get that,
\begin{align}
E-\frac{\Gamma(\theta)}{c}\frac{2 y^2}{(1+y)^3}L_{12}(\eps)(0)=\Big(\frac{\lambda_s}{\lambda}+1\Big)(F-y\pr_y F)-\frac{\Gamma(\theta)}{c}\frac{2 y^2}{(1+y)^3}L_{12}(\frac{3}{1+y}\sin(2\theta)\partial_\theta\eps)(0).
\end{align}
Since, $F=F_*+\alpha^2g$ with $F_*=\frac{4\alpha\Gamma y}{c(1+y)^2}$ and $F_*-y\pr_y F_*=\frac{8\alpha y^2}{(1+y)^2}$  we deduce that
\begin{align}
&E-\frac{\Gamma(\theta)}{c}\frac{2 y^2}{(1+y)^3}L_{12}(\frac{3}{1+y}\sin(2\theta)\partial_\theta\eps)(0)=\frac{4\Gamma y^2}{c(1+y)^3}\Big(\alpha\Big(\frac{\lambda_s}{\lambda}+1\Big)-L_{12}(\frac{3}{1+y}\sin(2\theta)\partial_\theta\eps)(0)\Big)\nonumber\\
&+\frac{\alpha^2}{2}\Big(\frac{\lambda_s}{\lambda}+1\Big)(g-y\pr_y g).
\end{align}
Hence, by using \eqref{goodboundonlambda} and \eqref{roughboundonlambda} we deduce that
\begin{align}
\Big|\Big(E-\frac{\Gamma(\theta)}{c}\frac{2 y^2}{(1+y)^3}L_{12}(\frac{3\sin(2\tet)}{1+y}\pr_\tet\eps)(0),\eps\Big)_{\mathcal{H}^k}\Big|&\lesssim \|\eps\|_{\mathcal{H}^k}\Big(\sqrt{\alpha}\|\eps\|_{\mathcal{H}^2}+\alpha\|\eps\|_{\mathcal{H}^2}^2\Big)\nonumber\\
&+\alpha\|g\|_{\mathcal{H}^k}\|\eps\|_{\mathcal{H}^k}^2.
\end{align}

\subsubsection*{Nonlinear terms}

We have that
\[|(y\partial_y\eps,\eps)_{\mathcal{H}^k}|=|(D_y \eps,\eps)_{\mathcal{H}^k}|\leq \frac{C}{\sqrt{\alpha}}|\eps|_{\mathcal{H}^k}^2,\qquad \Big|\frac{\mu_s}{\mu}\Big|\leq \frac{C}{\alpha}|\eps|_{\mathcal{H}^k},\] where the bound on $\frac{\mu_s}{\mu}$ comes from Proposition \ref{proposition:derivationoflaw}. In particular, 
\[\Big|\frac{\mu_s}{\mu}\Big|(y\partial_y\eps,\eps)_{\mathcal{H}^k}\leq \frac{C}{\alpha^{3/2}}|\eps|_{\mathcal{H}^k}^3.\]
Using identical reasoning and recalling that $\mathcal{S}_\delta(\eps)=\eps+(1+\delta)y\partial_y\eps$, we get:
\[\Big|\frac{\lambda_s}{\lambda}+1\Big||(S_\delta(\eps),\eps)_{\mathcal{H}^k}|\leq \frac{C}{\alpha^{3/2}}|\eps|_{\mathcal{H}^k}^3.\]

Next, \[(N_2(\eps),\eps)_{\mathcal{H}^k}=(-\frac{U_{\Phi_\eps}}{\sin(2\theta)}D_\theta\eps,\eps)_{\mathcal{H}^k}-\alpha(V(\Phi_\eps)D_y\eps,\eps)_{\mathcal{H}^k}+(\mathcal{R}(\Phi_\eps)\eps,\eps)_{\mathcal{H}^k}.\]
Observe that using Theorem \ref{Elliptic} and separating the $L_{12}$ part, we have
\[\Big|\frac{U_{\Phi_\eps}}{\sin(2\theta)}\Big|_{\mathcal{H}^k}+\Big|V_{\Phi_\eps}\Big|_{\mathcal{H}^k}+|\mathcal{R}(\Phi_\eps)|_{\mathcal{H}^k}\leq \frac{C}{\alpha}|\eps|_{\mathcal{H}^k} \]
In particular, using the product and transport estimates from Section \ref{ProductRules}, we get\footnote{This can be improved to $\alpha^{-1}$, but it is not important.}:
\[|(N_2(\eps),\eps)_{\mathcal{H}^k}|\leq \frac{C}{\alpha^{3/2}}|\eps|_{\mathcal{H}^k}^3.\]

Finally, we need to look at $N_3(\mathcal{U}^\phi)=2\mathcal{U}^\phi\tan(\theta)\alpha D_y \mathcal{U}^\phi+2\mathcal{U}^\phi\partial_\theta \mathcal{U}^\phi$. It is clear that \[|N_3(\mathcal{U}^\phi)|_{\mathcal{H}^k}\leq \frac{C}{\alpha^{1/2}}(|\tan(\theta)\mathcal{U}^\phi|_{\mathcal{H}^k}+|\partial_\theta \mathcal{U}^\phi|_{\mathcal{H}^k})|\mathcal{U}^\phi|_{\mathcal{H}^{k+1}}\leq \frac{C}{\alpha^{1/2}}\mathcal{E}^{2}.\]
In conclusion, we see that
\[\frac{d}{ds}(\eps,\eps)_{\mathcal{H}^k}\leq -c \|\eps\|_{\mathcal{H}^k}^2+\frac{C}{\alpha^{3/2}}\mathcal{E}^3.\]

\subsection{Bound on $\frac{d}{ds}(\mathcal{U}^\phi,\mathcal{U}^\phi)_{\mathcal{H}^{k+1}_{U_0}}$}

\[\frac{1}{2}\frac{d}{ds}(\mathcal{U}^\phi,\mathcal{U}^\phi)_{\mathcal{H}^{k+1}_{U_0}}\leq-(\mathcal{M}_F^\phi U^\phi,U^\phi)_{\mathcal{H}^{k+1}_{U_0}}+\Big|\frac{\mu_s}{\mu}\Big| |(y\partial_y \mathcal{U}^\phi, \mathcal{U}^\phi)_{\mathcal{H}^{k+1}_{U_0}}|+\Big|\frac{\lambda_s}{\lambda}+1\Big||(\mathcal{S}_\delta(\mathcal{U}_\phi),\mathcal{U}^\phi)_{\mathcal{H}^{k+1}_{U_0}}|+|(N_1(\Phi_\eps, \mathcal{U}^\phi),\mathcal{U}^\phi)_{\mathcal{H}^{k+1}_{U_0}}|.\]
Now, we know from \eqref{Mphi1} that 
\[(\mathcal{M}_F^\phi U^\phi,U^\phi)_{\mathcal{H}^{k+1}_{U_0}}\geq c|\mathcal{U}^\phi|_{\mathcal{H}^{k+1}}^2.\]
Moreover, as before, we have:
\[\Big|\frac{\mu_s}{\mu}\Big| |(y\partial_y \mathcal{U}^\phi, \mathcal{U}^\phi)_{\mathcal{H}^{k+1}_{U_0}}|+\Big|\frac{\lambda_s}{\lambda}+1\Big||(\mathcal{S}_\delta(\mathcal{U}_\phi),\mathcal{U}^\phi)_{\mathcal{H}^{k+1}_{U_0}}|\leq \frac{C}{\alpha^{3/2}}\mathcal{E}^3.\]
Now let us engage with the term $|(N_1(\Phi_\eps, \mathcal{U}^\phi),\mathcal{U}^\phi)_{\mathcal{H}^{k+1}_{U_0}}|$. We need to be a little careful with this term since we are only allowed to put $\eps$ in $\mathcal{H}^{k}$. What saves us is that we have bounds on $\tan(\theta) U^\phi$ and $\partial_\theta U^\phi$ and not just $U^\phi$.
\[|(N_1(\Phi_\eps, \mathcal{U}^\phi),\mathcal{U}^\phi)_{\mathcal{H}^{k+1}_{U_0}}|\leq 
|(U(\Phi_\eps)\partial_\theta \mathcal{U}^\phi, \mathcal{U}^\phi)_{\mathcal{H}^{k+1}_{U_0}}|+
|(V(\Phi_\eps)D_y \mathcal{U}^\phi, \mathcal{U}^\phi)_{\mathcal{H}^{k+1}_{U_0}}+
|(R(\Phi_\eps)\mathcal{U}^\phi,\mathcal{U}^\phi)_{\mathcal{H}^{k+1}_{U_0}}|+|(V(\Phi_\eps)\mathcal{U}^\phi,\mathcal{U}^\phi)_{\mathcal{H}^{k+1}_{U_0}}|\]

Observe that by a direct computation, we have: \[|(U(\Phi_\eps)\partial_\theta \mathcal{U}^\phi, \mathcal{U}^\phi)_{\mathcal{H}^{k+1}_{U_0}}|\leq C|\mathcal{U}^\phi|_{\mathcal{H}^{k+1}}|\partial_\theta \mathcal{U}^\phi|_{L^\infty}|U(\Phi_\eps)|_{\mathcal{H}^{k+1}}+\frac{C}{\alpha^{3/2}}|\mathcal{U}^\phi|_{\mathcal{H}^{k+1}}^2|\partial_\theta U(\Phi_\eps)|_{\mathcal{H}^k}\leq \frac{C}{\alpha^{3/2}}\mathcal{E}^3,\] since $\Phi_\eps$ contains a $\frac{1}{\alpha}$ in it and using the embedding proven in Proposition \ref{infinity_separated}. Similarly, 
\[|(V(\Phi_\eps)D_y \mathcal{U}^\phi, \mathcal{U}^\phi)_{\mathcal{H}^{k+1}_{U_0}}|\leq \frac{C}{\alpha^{3/2}}\mathcal{E}^3.\] 

Next, we need to look carefully at $|(R(\Phi_\eps),\mathcal{U}^\phi)_{\mathcal{H}^{k+1}_{U_0}}|$:
\[|(R(\Phi_\eps),\mathcal{U}^\phi)_{\mathcal{H}^{k+1}_{U_0}}|=2\Big|\Big((\tan(\theta)\Phi_\eps+\alpha\tan(\theta)D_y\Phi_\eps+\partial_\theta\Phi_\eps) \mathcal{U}^\phi, \mathcal{U}^\phi\Big)_{\mathcal{H}^{k+1}_{U_0}}\Big|\] The only non-trivial term is:
\[|(\tan(\theta)D_y\Phi_\eps \mathcal{U}^\phi, \mathcal{U}^\phi)_{\mathcal{H}^{k+1}_{U_0}}
|\leq C(|D_y\Phi_\eps|_{\mathcal{H}^{k+1}}|\tan(\theta)\mathcal{U}^\phi|_{L^\infty}|\mathcal{U}^\phi|_{\mathcal{H}^{k+1}}+\frac{C}{\alpha^{3/2}}|\mathcal{U}^\phi|_{\mathcal{H}^{k+1}}^2|\eps|_{\mathcal{H}^{k}})\leq \frac{C}{\alpha^{3/2}}\mathcal{E}^3.\]
Putting the above together, we see that:
\[\frac{d}{ds}(\mathcal{U}^\phi,\mathcal{U}^\phi)_{\mathcal{H}^{k+1}_{U_0}}\leq -c\|\mathcal{U}^\phi\|_{\mathcal{H}^{k+1}}^2+C\mathcal{E}^3.\]


\subsection{Bounds on $\frac{d}{ds}(\partial_\theta\mathcal{U}^\phi, \partial_\theta\mathcal{U}^\phi)_{\mathcal{H}^{k}_{U_1}}$ and $\frac{d}{ds}(\tan(\theta)\mathcal{U}^\phi, \tan(\theta)\mathcal{U}^\phi)_{\mathcal{H}^{k}_{U_2}}$.}\label{SwirlParts}
The non-linear estimates here are very similar to the above, so we omit them. We have to be very careful about the linear estimates, however, since $\partial_\theta$ and $\tan(\theta)$ \emph{do not} commute with the linear operator $\mathcal{M}_F^\phi$. The important fact is that the commutator will have a good sign in both cases. Let us recall (from Subsection \ref{MPhi}): \[\mathcal{M}^\phi_F (f)=\mathcal{M}^\phi_F (f)=y+y\partial_y f-\frac{3}{1+y}\sin(2\theta)\partial_\theta f+\frac{4-6\sin^2(\theta)}{1+y}f+l.o.t.\]

\[\partial_\theta \mathcal{M}^\phi_F(f)=  y\partial_\theta f+y\partial_y f+\frac{3}{1+y}\sin(2\theta)\partial_\theta \partial_\theta f+\frac{(4-6\sin^2(\theta))}{1+y}\partial_\theta f-\frac{6}{1+y}\cos(2\theta)\partial_\theta f-\frac{6\sin(2\theta)}{1+y} f +l.o.t. \] 
\[=  y\partial_\theta f+y\partial_y f+\frac{3}{1+y}\sin(2\theta)\partial_\theta \partial_\theta f+\frac{(-2+6\sin^2(\theta))}{1+y}\partial_\theta f-\frac{6\sin(2\theta)}{1+y} f +l.o.t. \] 
\[=\mathcal{L}_{F_*}(\partial_\theta f)+\frac{6\sin^2(\theta)}{1+y}\partial_\theta f+ l.o.t.\]
It is then easy to show, as before, that we can find an inner product $(\cdot,\cdot)_{\mathcal{H}^k_{U_1}}$ so that 
\[(\partial_\theta \mathcal{M}^\phi_F(f), \partial_\theta f)_{\mathcal{H}^k_{U_1}}\geq c|\partial_\theta \mathcal{U}^\phi|_{\mathcal{H}^k}^2.\] 
It then follows that
\[\frac{d}{ds}(\partial_\theta\mathcal{U}^\phi,\partial_\theta\mathcal{U}^\phi)_{\mathcal{H}^k_{U_1}}\leq -c|\partial_\theta \mathcal{U}^\phi|_{\mathcal{H}^k}^2+\frac{C}{\alpha^{3/2}}\mathcal{E}^3.\]

Now we check what happens when we multiply by $\tan(\theta)$. In this case we get:
\[\tan(\theta) \mathcal{M}^\phi_F(f)= \mathcal{M}^\phi_F(\tan(\theta)f)+\frac{3}{1+y}\sin(2\theta)\sec^2(\theta)f+l.o.t.\]
\[=\mathcal{M}^\phi_F(\tan(\theta)f)+\frac{6}{1+y}\tan(\theta)f+l.o.t.\] This again implies that 
\[(\tan(\theta)\mathcal{M}^\phi_F(f), \tan(\theta)f)_{\mathcal{H}^{k}_{U_2}}\geq c|f|_{\mathcal{H}^k}^2,\] with the inner product giving a norm equivalent to the $\mathcal{H}^k$ norm. Thus it is not difficult to see that
\[\frac{d}{ds}(\tan(\theta)\mathcal{U}^\phi,\tan(\theta)\mathcal{U}^\phi)_{\mathcal{H}^k_{U_2}}\leq -c|\tan(\theta) \mathcal{U}^\phi|_{\mathcal{H}^k}^2+\frac{C}{\alpha^{3/2}}\mathcal{E}^3.\] 
\subsection{Final Estimate}

Putting together what we gained from the preceding calculations and defining \[\bar{\mathcal{E}}=(\eps,\eps)_{\mathcal{H}^k}+(\mathcal{U}^\phi,\mathcal{U}^\phi)_{\mathcal{H}^{k+1}_{U_0}}+(\partial_\theta\mathcal{U}^\phi,\partial_\theta\mathcal{U}^\phi)_{\mathcal{H}^k_{U_1}}+(\tan(\theta)\mathcal{U}^\phi,\tan(\theta)\mathcal{U}^\phi)_{\mathcal{H}^{k}_{U_2}},\] we get:
\[\frac{d}{ds} 
\bar{\mathcal{E}}
\leq 
-c\bar{\mathcal{E}}+
\frac{C}{\alpha^{3/2}}\bar{\mathcal{E}}^{3/2},\] with $c$ and $C$ independent of $\alpha$. Note that $\bar{\mathcal{E}}\approx \mathcal{E}^2$ with constants independent of $\alpha$. Theorem \ref{StabilityTheorem} follows directly from this bound.

\section{Acknowledgements}
T.M.E. acknowledges funding from the NSF DMS-1817134. T. G and N. M are partially supported by SITE (Center for Stability, Instability, and Turbulence at NYUAD).  N. M is partially supported by  NSF DMS-1716466.  

\appendix
\section{Appendix: Product Rules}\label{ProductRules}

We now move to establish energy estimates in $\mathcal{H}^k$. First recall from \cite{E_Classical} that $\mathcal{H}^k$ embeds continuously in $L^\infty$. Let us also observe the following lemma that implies the embedding. 

\begin{lemma}\label{infinity_separated}For all $z$, we have
\[\sup_\theta |g(z,\theta)|^2\leq \frac{C}{\sqrt{\gamma-1}} \int_0^{\pi/2} |\partial_\theta g(z,\theta)|^2 \sin(2\theta)^{2-\gamma}d\theta.\] And for all $\theta$ we have
\[\sup_z |g(z,\theta)|^2\leq C\int_0^\infty |D_z g(z,\theta)|^2 \frac{(1+z)^4}{z^4}dz.\]
\end{lemma} 


As a consequence, we have the following result.

\begin{proposition}\label{P1}[First Product Rule]
Let $k\geq 4$ and assume $f,g\in \mathcal{H}^k$. Then, $fg\in \mathcal{H}^k$ and \[|fg|_{\mathcal{H}^k}\leq \frac{C}{\sqrt{\gamma-1}}|f|_{\mathcal{H}^k}|g|_{\mathcal{H}^k}.\]
\end{proposition}
Let us focus our discussion on the case $k=4$. Most of the terms can be controlled trivially using the embedding of $\mathcal{H}^k$ in $L^\infty$ (except one). The point is that if we take four derivatives of the product $fg$, one of the two must always have at most two derivatives on it and in that case we just put it in $L^\infty$ and pull it out of the integral. The only awkward term to handle is the following (due to the discrepancy in the angular weights):
\[I:=\int \int (D_z^3 f)^2 (\partial_\theta g)^2 \sin(2\theta)^{2-\gamma}\frac{(1+z)^4}{z^4}dzd\theta.\] To handle this term, we make use of Lemma \ref{infinity_separated}. Indeed,
\[I \leq \Big(\int_0^{\pi/2} \sup_z |\partial_\theta g|^2\sin(2\theta)^{2-\gamma} d\theta\Big)\Big(\int_0^\infty \sup_\theta |D_z^3 f|^2 \frac{(1+z)^4}{z^4}dz\Big).\] Thus, by Lemma \eqref{infinity_separated}, we get:
\[I\leq C |g|_{\mathcal{H}^4}^2 \cdot \frac{C}{\gamma-1} |f|_{\mathcal{H}^4}^2. \]

As in \cite{E_Classical}, we need a second product rule. 
\begin{proposition}\label{P2} [Second Product Rule]
Let $f\in \mathcal{H}^k$ and $g\in \mathcal{W}^{k,\infty}$. Then, $fg\in \mathcal{H}^k$ and\[|fg|_{\mathcal{H}^k}\leq \frac{C}{\sqrt{\gamma-1}}|f|_{\mathcal{H}^k}|g|_{\mathcal{W}^{k,\infty}}.\]
\end{proposition}

\begin{proof}
As in \cite{E_Classical}, whenever derivatives fall onto $g$ we can take them out of the integral (this is easy to see in $I$ above, for example). 
\end{proof}

\begin{proposition}\label{T1} [First Transport Estimate]
Assume $k\geq 4$ and that $u,v, g\in \mathcal{H}^k.$ Then,
\[|(u D_\theta g, g)_{\mathcal{H}^k}|\leq \frac{C}{\sqrt{\gamma-1}}|u|_{\mathcal{H}^k}|g|_{\mathcal{H}^k}^2\] and 
\[|(vD_z g, g)_{\mathcal{H}^k}|\leq \frac{C}{\sqrt{\gamma-1}}|v|_{\mathcal{H}^k}|g|_{\mathcal{H}^k}^2\]
\end{proposition}

\begin{proof}
This is essentially an application of the product rules.  It is important to note that, since $k\geq 4$, we have $|D^j g|_{L^\infty}\leq \frac{C}{\sqrt{\gamma-1}}|g|_{\mathcal{H}^k},$ whenever $j\leq k-2$. 
\end{proof}

Similarly, we have the second transport estimate. 

\begin{proposition}\label{T2} [Second Transport Estimate]
Assume $k\geq 4$ and that $D_\theta u,D_zv\in \mathcal{W}^{k,\infty}$ and $g\in \mathcal{H}^k$ Then,
\[|(u D_\theta g, g)_{\mathcal{H}^k}|\leq \frac{C}{\sqrt{\gamma-1}}| u|_{\mathcal{W}^{k,\infty}}|g|_{\mathcal{H}^k}^2\] and 
\[|(vD_z g, g)_{\mathcal{H}^k}|\leq \frac{C}{\sqrt{\gamma-1}}| v|_{\mathcal{W}^{k,\infty}}|g|_{\mathcal{H}^k}^2.\]
\end{proposition}

For the cut-off procedure we also make use of the following lemma.
\begin{lemma}\label{RadialMultiplier}
Let $\phi\in C^k([0,\infty))$ and assume that $\sup_{0\leq j\leq k}|D_z^j\phi|_{L^\infty}=C_{\phi}<\infty.$ Then, if $f\in\mathcal{H}^k$, we have that $\phi f\in\mathcal{H}^k$ and \[|f\phi|_{\mathcal{H}^k}\leq C_kC_\phi |f|_{\mathcal{H}^k}.\]
\end{lemma}
\begin{remark}
The important point is that we do not get the $\frac{1}{\sqrt{\gamma-1}}$ loss in this estimate. 
\end{remark}
\begin{remark}
Observe that if $\chi$ is a radial smooth cut-off function, if we consider $\chi^M(z):=\chi(\frac{z}{M})$, then $|D_z^j (\chi^M)|_{L^\infty}\leq C_j,$ for all $j\leq k$ (the bound is independent of $M$). 
\end{remark}

\begin{proof}
Observe that if $D^j$ consists of $j$ derivatives (either $D_\theta$ or $D_z$), let us let $W_{D^j}$ denote either the weight $\sin(2\theta)^{-\eta}\frac{(1+z)^4}{z^4}$ or $\sin(2\theta)^{-\gamma}\frac{(1+z)^4}{z^4}$ depending on whether $D^j$ contains a $D_\theta$ or not.  Observe also that $D_\theta$ derivatives cannot hit $\phi$ since $\phi$ is radial. As a consequence, 
\[\int |D^j(f\phi)|^2 W_{D^j}\leq C_j C_\phi\sum_{|\beta|<j} |D^{\beta}f|^2 W_{D^\beta}\leq C_j C_\phi |f|_{\mathcal{H}^k}.\] Now summing over all $|j|\leq k$ we get the result. 
\end{proof}

\bibliographystyle{plain}
\bibliography{3dEuler_away_final}

\begin{thebibliography}{10}

\bibitem{BardosTitiReview}
Claude Bardos and Edriss~S. Titi.
\newblock Euler equations for incompressible ideal fluids.
\newblock {\em Uspekhi Mat. Nauk}, 62(3 (375)):5--46, 2007.

\bibitem{Chae2007}
Dongho Chae.
\newblock Nonexistence of self-similar singularities for the 3d incompressible
  {E}uler equations.
\newblock {\em Comm. Math. Phys.}, 273(1):203--215, 2007.

\bibitem{ChaeShv}
Dongho Chae and Roman Shvydkoy.
\newblock On formation of a locally self-similar collapse in the incompressible
  {E}uler equations.
\newblock {\em Arch. Ration. Mech. Anal.}, 209(3):999--1017, 2013.

\bibitem{ChenHou}
J.~Chen and T.Y. Hou.
\newblock Finite time blowup of 2d {B}oussinesq and 3d {E}uler equations with
  ${C}^{1,\alpha}$ velocity and boundary.
\newblock {\em ArXiv e-prints}, 2019.

\bibitem{CGM}
C.~Collot, T.E. Ghoul, and N.~Masmoudi.
\newblock Singularity formation for burgers equation with transverse viscosity.
\newblock {\em ArXiv e-prints}, 2018.

\bibitem{ConstantinReview}
P.~Constantin.
\newblock On the {E}uler equations of incompressible fluids.
\newblock {\em Bull. Amer. Math. Soc. (N.S.)}, 44(4):603--621, 2007.

\bibitem{EGM}
T.~M. Elgindi, T.~D. Ghoul, and N~Masmoudi.
\newblock Stable self-similar blowup for a family of nonlocal transport
  equations.
\newblock {\em ArXiv e-prints}, 2019.

\bibitem{EJB}
T.~M. {Elgindi} and I.-J. {Jeong}.
\newblock {Finite-time Singularity Formation for Strong Solutions to the
  Boussinesq System}.
\newblock {\em ArXiv e-prints}, August 2017.

\bibitem{E_Classical}
Tarek~M. Elgindi.
\newblock Finite-time singularity formation for ${C}^{1,\alpha}$ solutions to
  the 3d incompressible euler equations.
\newblock {\em ArXiv e-prints}, 2019.

\bibitem{EJE}
Tarek~M. Elgindi and In-Jee Jeong.
\newblock Finite-time singularity formation for strong solutions to the
  axi-symmetric $3d$ {E}uler equations.
\newblock {\em Ann. PDE}, 2019.

\bibitem{Elling}
Volker Elling.
\newblock Self-similar 2d {E}uler solutions with mixed-sign vorticity.
\newblock {\em Comm. Math. Phys.}, 348(1):27--68, 2016.

\bibitem{MR3779644}
Tej-Eddine Ghoul, Van~Tien Nguyen, and Hatem Zaag.
\newblock Blowup solutions for a reaction-diffusion system with exponential
  nonlinearities.
\newblock {\em J. Differential Equations}, 264(12):7523--7579, 2018.

\bibitem{MR3846237}
Tej-Eddine Ghoul, Van~Tien Nguyen, and Hatem Zaag.
\newblock Construction and stability of blowup solutions for a non-variational
  semilinear parabolic system.
\newblock {\em Ann. Inst. H. Poincar\'{e} Anal. Non Lin\'{e}aire},
  35(6):1577--1630, 2018.

\bibitem{Gibbon2008}
J.~D. Gibbon.
\newblock The three-dimensional {E}uler equations: where do we stand?
\newblock {\em Phys. D}, 237(14-17):1894--1904, 2008.

\bibitem{JiaSverakSS}
Hao Jia and Vladim\'ir \v{S}ver\'ak.
\newblock Local-in-space estimates near initial time for weak solutions of the
  {N}avier-{S}tokes equations and forward self-similar solutions.
\newblock {\em Invent. Math.}, 196(1):233--265, 2014.

\bibitem{MR2257393}
Carlos~E. Kenig and Frank Merle.
\newblock Global well-posedness, scattering and blow-up for the
  energy-critical, focusing, non-linear {S}chr\"{o}dinger equation in the
  radial case.
\newblock {\em Invent. Math.}, 166(3):645--675, 2006.

\bibitem{KiselevReview}
A~Kiselev.
\newblock Small scales and singularity formaiton in fluid dynamics.
\newblock {\em arXiv:1807.00184}.

\bibitem{HouLuo}
Guo Luo and Thomas~Y. Hou.
\newblock Toward the finite-time blowup of the 3{D} axisymmetric {E}uler
  equations: a numerical investigation.
\newblock {\em Multiscale Model. Simul.}, 12(4):1722--1776, 2014.

\bibitem{MB}
Andrew~J. Majda and Andrea~L. Bertozzi.
\newblock {\em Vorticity and incompressible flow}, volume~27 of {\em Cambridge
  Texts in Applied Mathematics}.
\newblock Cambridge University Press, Cambridge, 2002.

\bibitem{MR3179608}
Yvan Martel, Frank Merle, and Pierre Rapha\"{e}l.
\newblock Blow up for the critical generalized {K}orteweg--de {V}ries equation.
  {I}: {D}ynamics near the soliton.
\newblock {\em Acta Math.}, 212(1):59--140, 2014.

\bibitem{MR2150386}
Frank Merle and Pierre Raphael.
\newblock The blow-up dynamic and upper bound on the blow-up rate for critical
  nonlinear {S}chr\"{o}dinger equation.
\newblock {\em Ann. of Math. (2)}, 161(1):157--222, 2005.

\bibitem{MR1427848}
Frank Merle and Hatem Zaag.
\newblock Stability of the blow-up profile for equations of the type
  {$u_t=\Delta u+|u|^{p-1}u$}.
\newblock {\em Duke Math. J.}, 86(1):143--195, 1997.

\bibitem{MR3302641}
Frank Merle and Hatem Zaag.
\newblock On the stability of the notion of non-characteristic point and
  blow-up profile for semilinear wave equations.
\newblock {\em Comm. Math. Phys.}, 333(3):1529--1562, 2015.

\bibitem{NRS96}
J.~Necas, M.~Ruzicka, and V.~Sver\'ak.
\newblock On {L}eray's self-similar solutions of the {N}avier-{S}tokes
  equations.
\newblock {\em Acta Math.}, 176(2):283--294, 1996.

\bibitem{Tsai1998}
Tai-Peng Tsai.
\newblock On {L}eray's self-similar solutions of the {N}avier-{S}tokes
  equations satisfying local energy estimates.
\newblock {\em Arch. Rational Mech. Anal.}, 143(1):29--51, 1998.

\bibitem{VasseurVishik}
A.~Vasseur and M.~Vishik.
\newblock Blow-up solutions to 3{D} {E}uler are hydrodynamically unstable.
\newblock {\em ArXiv e-prints}, 2019.

\bibitem{VishikNon1}
M.~Vishik.
\newblock Instability and non-uniqueness in the {C}auchy problem for the
  {E}uler equations of an ideal incompressible fluid. {P}art {I}.
\newblock {\em ArXiv e-prints}, 2018.

\bibitem{VishikNon2}
M.~Vishik.
\newblock Instability and non-uniqueness in the {C}auchy problem for the
  {E}uler equations of an ideal incompressible fluid. {P}art {II}.
\newblock {\em ArXiv e-prints}, 2018.

\end{thebibliography}

\end{document}